\documentclass[final,reqno,onefignum,onetabnum]{siamart171218}
\usepackage{subdepth}
\usepackage{overpic}
\usepackage{amsmath,amstext,amsbsy,amssymb,mathdots}
\usepackage{booktabs,bm}
\usepackage{mathrsfs}
\usepackage{tikz,cite}
\usepackage{courier}
\usetikzlibrary{positioning}
\usetikzlibrary{shapes,arrows}
\usepackage[fleqn,tbtags]{mathtools}
\usepackage{amsfonts}
\usepackage{relsize}
\usepackage{xcolor,colortbl}
\usepackage{psfrag}
\usepackage{alltt}
\usepackage{arydshln}
\usepackage{enumerate}
\usepackage{epstopdf}
\usepackage{enumitem}
\usepackage{multirow}
\usepackage{algorithmic}
\usepackage{framed}
\usepackage{cancel}
\usepackage{booktabs}
\usepackage{xargs}

\usepackage[colorinlistoftodos,prependcaption,textsize=tiny]{todonotes}
\reversemarginpar
\newcommandx{\attodo}[2][1=]{\todo[linecolor=red,backgroundcolor=red!25,bordercolor=red,#1]{#2}}
\newcommandx{\mgtodo}[2][1=]{\todo[linecolor=blue,backgroundcolor=blue!25,bordercolor=blue,#1]{#2}}
\newcommandx{\agreedtodo}[2][1=]{\todo[linecolor=green,backgroundcolor=green!25,bordercolor=green,#1]{#2}}

\colorlet{lightgray}{gray!40}

\newcommand{\<}{\left\langle }
\let\>\undefined
\newcommand{ \>}{\right\rangle }
\DeclareMathOperator*{\argmin}{arg\,min}

\newcommand{\mV}{\mathcal{V} }
\newcommand{\mW}{\mathcal{W} }
\newcommand{\mL}{\mathcal{L} }
\newcommand{\mR}{\mathcal{R}}
\newcommand{\K}{\mathcal{K} }
\newcommand{\B}{\mathcal{B} }
\newcommand{\Hone}{\smash{\mathcal{H}^1(\Omega)}}
\newcommand{\Htwo}{\smash{\mathcal{H}^2(\Omega)}}
\newcommand{\Htwozero}{\smash{\mathcal{H}^1_0(\Omega)\cap\mathcal{H}^2(\Omega)}}

\newcommand{\Honezero}{\smash{\mathcal{H}_0^1(\Omega)}}
\newcommand{\Ltwo}{\smash{L^2(\Omega)}}
\newcommand{\PiLPi}{\smash{\Pi_{\mV_0}^*\mathcal{L}\Pi_{\mV_0}}}

\newsiamremark{assumption}{Assumption} % <- Preamble
\crefname{assumption}{Assumption}{Assumptions} % <- Preamble
\newsiamremark{problem}{Problem} % <- Preamble
\crefname{problem}{Problem}{Problems} % <- Preamble

\newcolumntype{"}{@{\hskip\tabcolsep\vrule width 1pt\hskip\tabcolsep}}
\setlist[description]{font=\normalfont\space}

\title{Continuous analogues of Krylov methods for differential operators\thanks{Submitted to the editors \today.
\funding{This work is supported by National Science Foundation grant No.~1645445.}}}
\author{Marc Aur\`ele Gilles\thanks{Center for Applied Mathematics, Cornell University, Ithaca, NY 14853. (\email{mtg79@cornell.edu})} \and Alex Townsend\thanks{Department of Mathematics, Cornell University, Ithaca, NY  14853. (\email{townsend@cornell.edu})}}
\headers{Krylov Subspace Methods for Differential Operators}{Marc Aur\`ele Gilles and Alex Townsend}
\begin{document}
\maketitle

\begin{abstract}
Analogues of the conjugate gradient method, MINRES, and GMRES are derived for solving boundary value problems (BVPs) involving second-order differential operators. Two challenges arise: imposing the boundary conditions on the solution while building up a Krylov subspace, and guaranteeing convergence of the Krylov-based method on unbounded operators. Our approach employs projection operators to guarantee that the boundary conditions are satisfied, and we develop an operator preconditioner that ensures that an approximate solution is computed after a finite number of iterations. The developed Krylov methods are practical iterative BVP solvers that are particularly efficient when a fast operator-function product is available.
\end{abstract}



\section{Introduction}\label{sec:introduction}
Krylov methods, such as the conjugate gradient (CG) method, MINRES, and GMRES, are iterative algorithms that solve $Ax=b$ using matrix-vector products~\cite{van2003iterative}. After $k$ iterations, they typically compute an approximate solution to $Ax=b$ from the Krylov subspace $\smash{\mathcal{K}_k(A,b) = {\rm Span}\{b,Ab,\ldots,A^{k-1}b\}}$. They provide a toolkit for solving large sparse or structured linear systems, which arise in the discretization of differential equations~\cite{elman2014finite}.

In this paper, we describe operator analogues of Krylov methods for solving two-point boundary value problems (BVPs)~\cite[Chap.~6]{evans2010partial} that avoid discretizing the differential operator. Our Krylov methods directly solve
\begin{equation}
\mathcal{L}u = f\quad \text{on } \Omega =(-1,1), \qquad u(\pm 1) = 0,
\label{eq:PDE}
\end{equation}
where $\smash{\mathcal{L}u = -(a(x)u'(x))' + b(x)u'(x) + c(x) u}$, $\mathcal{L}:\Htwozero\rightarrow L^2( \Omega)$\footnote{ This is a slight restriction from the more typical $\mathcal{L}:\mathcal{H}^1_0( \Omega) \rightarrow \mathcal{H}^{-1}( \Omega)$ setup. However, this restriction is natural in the present context where we develop practical algorithms which apply $\mathcal{L}$ to functions by weak differentiation operations and function products instead of having to revert to a bilinear form interpretation of the function product (see~\cref{sec:CG1dAlgorithm}). }, $a,b,c\in L^\infty(\Omega)$ and $f\in L^2(\Omega)$.

If there are no additional assumptions on $\mathcal{L}$, then we propose an analogue of GMRES to solve~\cref{eq:PDE} (see \cref{sec:GMRES}). If $b(x)=0$, then $\mathcal{L}$ is self-adjoint, which is analogous to a symmetric matrix, and we propose an analogue of MINRES (see \cref{sec:MINRES}). When $a(x)>0$, $b(x)=0$, and $c(x)\geq 0$, $\mathcal{L}$ is self-adjoint with real positive eigenvalues~\cite[Sec.~6.5, Thm.~1]{evans2010partial}, which is analogous to a symmetric positive definite matrix, and we propose an analogue of the CG method (see \cref{sec:CGtheory}).

To see the difficulties in developing a Krylov-based method for differential operators, consider solving $-u''(x) = 1-x^2$ on $\Omega =(-1,1)$ with $u(\pm 1) = 0$. The exact solution is $u (x) = (x^4 - 6x^2 + 5)/12$. A naive generalization of the Krylov subspace is $\mathcal{K}_k(\mL, f) = {\rm Span} \{ f , \mL f, \mL\left( \mL f \right) , \dots, \mL^{k-1}f \}$ with $f = 1-x^2$. Since $\mL u = -u'' $, this leads to $\mathcal{K}_k( \mL, f) = {\rm Span}\{ 1-x^2, 2 \}$ for $k\geq 2$. This example illustrates that such an approach is flawed, as $\mathcal{K}_k( \mL, f)$ does not contain a good approximation to the exact solution. Moreover, the boundary conditions are not imposed because $\mathcal{K}_k( \mL,f)\not\subset \mathcal{H}^1_0(\Omega)$ for $k\geq 2$. There are at least three major theoretical issues to overcome:
\begin{description}[leftmargin=*,noitemsep]
\item \begin{problem}\label{problem1} Since $f\not\in \mathcal{H}^1_0(\Omega)$ and ${\rm Range}(\mL) \not\subset \mathcal{H}^1_0(\Omega)$, how does one construct a Krylov subspace that satisfies the boundary conditions? Our answer involves using orthogonal projection operators to ensure that each term in the Krylov subspace is in $\Honezero$ (see \cref{sec:RestrictedCG}), and solving an ancillary problem (see \cref{sec:GeneralRHS}). \end{problem}
\item \begin{problem}\label{problem2} Since $\mathcal{L}: \Htwozero \rightarrow L^2(\Omega)$, how does one repeatedly apply operator-function products that are necessary to build up a Krylov subspace? To achieve this, we use an orthogonal projection operator and a preconditioner that acts as a ``smoother" (see \cref{sec:preconditioning}).
\end{problem}
\item \begin{problem}\label{problem3} Since $\mathcal{L}$ is an unbounded operator, how does one construct a Krylov method that rapidly converges to the solution of~\cref{eq:PDE}? Our answer is to use operator preconditioners that allow for our Krylov iterations to be terminated after a finite number of iterations with an approximate solution (see \cref{sec:convergence}).
\end{problem}
\end{description}

The standard discretization-based approach is problematic for spectral methods. Despite $n\times n$ Chebyshev-based spectral discretization matrices of~\cref{eq:PDE} having a fast $\mathcal{O}(n\log n)$ matrix-vector product based on the FFT~\cite{mason2002chebyshev}, Krylov methods are not ubiquitously employed in the spectral method community~\cite{fornberg1998practical,trefethen2000spectral}. Since matrices derived from spectral methods are more ill-conditioned than expected (see \cref{fig:KrylovSolvers}), good low-order preconditioners can be difficult to design~\cite{orszag1980spectral}, especially for BVPs with variable coefficients~\cite{canuto1985preconditioned}.
Additionally, spectral discretizations based on Chebyshev polynomials are typically not structure-preserving, which can prohibit the use of the matrix CG method even when the underlying differential operator is self-adjoint with $a(x)>0$ and $c(x)\geq 0$.

The Krylov methods that we develop solve~\cref{eq:PDE} by directly applying $\mathcal{L}$ to functions, and we prove that the iterates from our preconditioned CG method geometrically converge to the solution (see \cref{cor:CGconvergence}). Our operator Krylov methods are not equivalent to matrix Krylov methods applied to a standard discretization of~\cref{eq:PDE}, and offer several advantages: (1) Operator preconditioners are motivated by the differential operator as opposed to the properties of a discretization scheme, (2) The resulting CG method can always be applied to~\cref{eq:PDE} with $a(x)>0$, $b(x)=0$, and $c(x)\geq 0$ without the need for structure-preserving discretizations~\cite[Chap.~4]{shen2011spectral}, and (3) The iterates converge to the desired solution of~\cref{eq:PDE}, as opposed to the solution of a discretization.

\begin{figure}
\centering
\begin{minipage}{.49\textwidth}
\begin{overpic}[width=\textwidth]{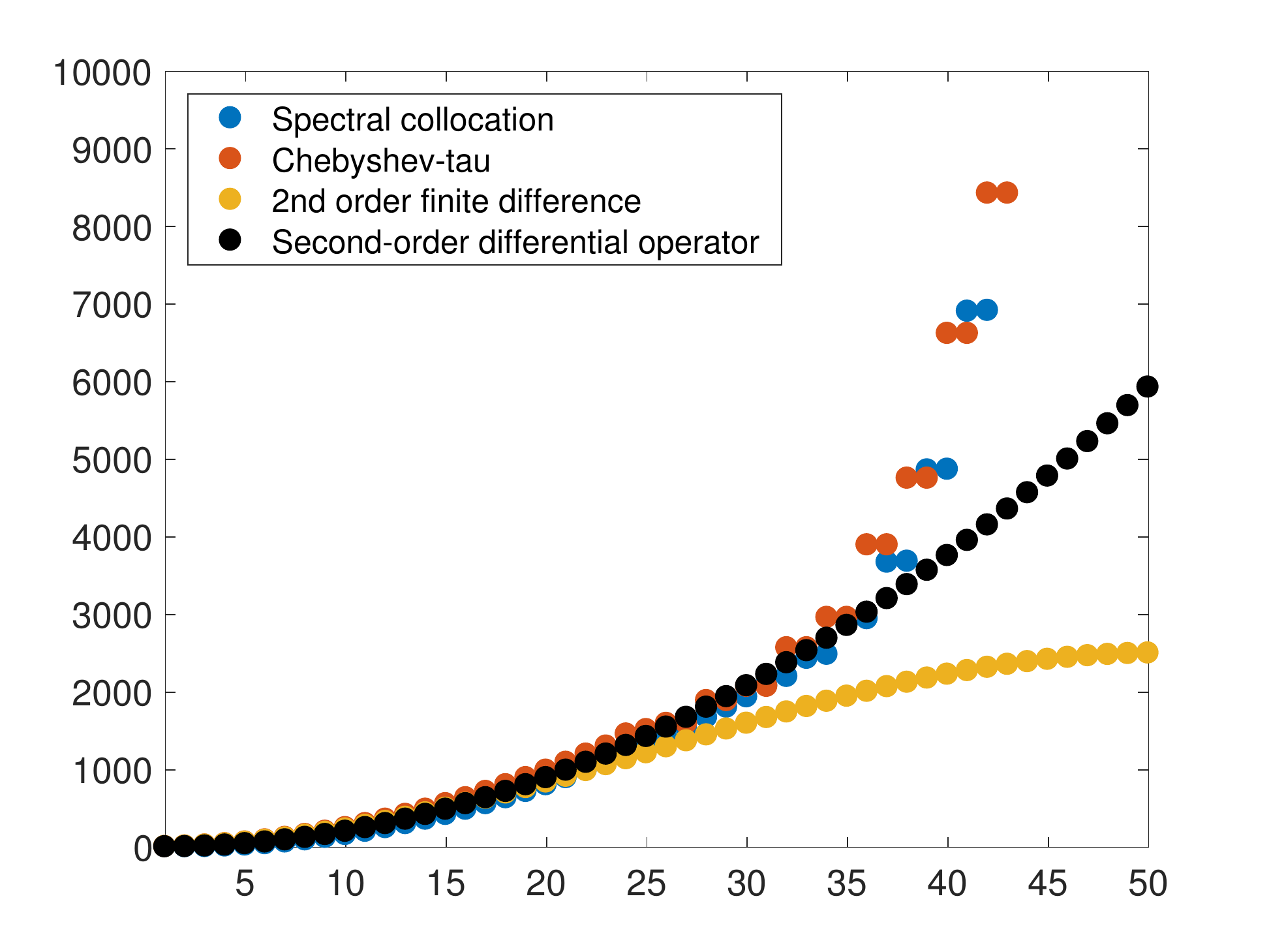}
\put(93,60) {\vector(0,1){10}}
\put(94,45) {\rotatebox{90}{more eigvals}}
\put(50,0) {$k$}
\put(0,35) {\rotatebox{90}{$|\lambda_k|$}}
\end{overpic}
\end{minipage}
\caption{Spectra of $50\times 50$ discretizations of $\mathcal{L} u = -u''$ with zero Dirichlet boundary conditions. Similar figures appear in~\cite[Fig.~1]{weideman1988eigenvalues}. Spectral collocation (blue dots)~\cite{trefethen2000spectral}, and Chebyshev tau (red dots)~\cite{ortiz1969tau} discretizations typically have spectra that grow asymptotically faster than the spectra of the underlying differential operator (black dots), while the spectra of finite difference (yellow dots)~\cite{leveque2007finite} discretizations grow asymptotically slower. Most popular spectral discretizations are more ill-conditioned than expected, leading to poor convergence properties of Krylov subspace solvers.  Our operator Krylov methods avoid discretizing BVPs and employs preconditioners that are motivated from the differential operator (see \cref{sec:preconditioning}).}
\label{fig:KrylovSolvers}
\end{figure}
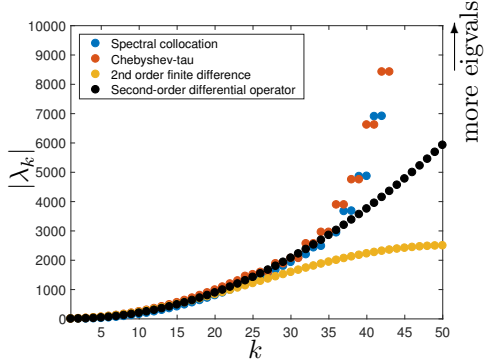

%A discussion at the Chebfun and Beyond conference in September 2012 attended by about 150 numerical analysts was on the topic of Krylov methods for differential operators. The session was chaired by Nick Higham. The second author was a graduate student at the time and scribed the discussion, following Nick Trefethen's advice. It was agreed at the discussion that designing Krylov methods for differential operators was important; particularly, in the context of spectral collocation methods. During that meeting, a promising differential GMRES method for computing oscillatory integrals was discussed~\cite{olver2009gmres}, but it has remained unclear how to successfully incorporate boundary conditions into that framework.
Several attempts to develop operator Krylov methods for differential equations have been proposed that we believe date back to 1967~\cite{daniel1967conjugate}, where it was shown that an operator CG method can be reduced to a sequence of Poisson problems with Dirichlet boundary conditions. In 2009, a promising differential GMRES method for computing oscillatory integrals~\cite{olver2009gmres} was developed in the context of spectral methods, but it has remained unclear how to successfully incorporate boundary conditions.  A theoretical foundation for a CG method on ordinary and partial differential operators~\cite{malek2014preconditioning} was introduced in 2015. The authors use a Riesz map $\tau: \mathcal{H}^{-1}(\Omega)\rightarrow \Honezero$ to precondition a differential operator~\cite[Chap.~4]{malek2014preconditioning} and successfully construct a Krylov subspace of the form ${\rm Span}\{\tau f, \tau \mathcal{L}\tau f, (\tau\mathcal{L})^2\tau f, \ldots,\}$. This work is an insightful theoretical framework and our paper expands on their contribution in order to develop a collection of practical Krylov methods for solving~\cref{eq:PDE}.

Though we do not discretize the differential operator itself, for our operator Krylov methods to be of practical interest, one must employ an approximation space for the solution and right-hand side (see \cref{sec:CG1dAlgorithm}). Unlike most BVP solvers, the approximation space can be all of $\Honezero$ or an infinite dimensional dense subspace of $\Honezero$. This allows one to implement highly adaptive Krylov subspace methods that automatically resolve the solution to machine precision (see \cref{sec:CG1dAlgorithm}).

Intuitively, our main idea is to modify the operator-function products with $\mathcal{L}$ while preserving the weak form of~\cref{eq:PDE}. That is, we respect the bilinear form~\cite[p.~316]{evans2010partial} associated with~\cref{eq:PDE}, i.e.,
\begin{equation}
\B[\phi,\psi] = \int_{-1}^1 a(x) \phi'(x)\psi'(x) + b(x)\phi'(x)\psi(x)+c(x)\phi(x)\psi(x)  dx, \qquad \phi,\psi\in \mathcal{H}_0^1(\Omega)
\label{eq:bilinearForm}
\end{equation}
as well as the weak form of the solution as $\B[u,\psi] = \<f,\psi\>$ for all $\psi\in\mathcal{H}^1_0(\Omega)$.  Here, and throughout the paper, we use $\<\cdot,\cdot\>$ to denote the standard $L^2$ inner-product and $\|\psi\|^2 = \<\psi,\psi\>$.

The paper is structured as follows. In \cref{sec:CGtheory} we derive an unpreconditioned and preconditioned CG method for solving~\cref{eq:PDE} when $\mL $ is a self-adjoint second-order differential operator with $a(x)>0$, $b(x) = 0$, and $c(x)\geq 0$. In \cref{sec:CG1dAlgorithm} we use our CG theory to develop practical iterative BVP solvers for~\cref{eq:PDE}. In \cref{sec:OtherKrylovBasedMethods}, we extend our CG method to operator analogues of MINRES and GMRES.

\section{The CG method for differential operators}\label{sec:CGtheory}
The CG method for matrices is an iterative algorithm for solving $Ax=b$, where $A$ is a symmetric positive definite matrix~\cite{hestenes1952methods}. It constructs iterates $x_0=0,x_1,x_2,\ldots,$ such that $x_k$ is the best approximate from $\mathcal{K}_k(A,b)$ as measured by the energy norm. That is,
\[
x_k = \argmin_{y \in \K_k(A,b) }\| x-y \|_A, \qquad \K_k(A,b) = {\rm Span}\left\{b,Ab,\ldots,A^{k-1}b\right\},
\]
where $\smash{\|y\|_A^2 = y^TAy}$ and $x = A^{-1}b$ is the exact solution.  The fact that $\|\cdot\|_A$ defines a norm is central to the development and analysis of the CG method for matrices~\cite[Sec.~5.6]{liesen2013krylov}.

Just like symmetric positive definite matrices, self-adjoint differential operators with $a(x)>0$ and $c(x)\geq 0$ have real positive eigenvalues and an orthogonal basis of eigenfunctions~\cite[Sec.~6.5, Thm.~1]{evans2010partial}. The analogue of the energy norm in this setting is $\|\phi\|_{\mathcal{L}}^2 = \B[\phi,\phi]$ for $\phi\in\Honezero$, where $\B$ is the bilinear form associated to $\mathcal{L}$ in~\cref{eq:bilinearForm}. The fact that $\|\cdot\|_{\mathcal{L}}$ defines a norm is equally important for the development and analysis of a CG method for~\cref{eq:PDE}.

If $p_0,p_1,\ldots,$ form a complete basis for $\mathcal{H}^1_0(\Omega)$ so that $\B[p_i,p_j] = 0$ for $i\neq j\geq 0$, then since $f\in \Ltwo$, we may formally write the solution to~\cref{eq:PDE} as
\[
u = \sum_{j=0}^\infty \frac{\<f,p_j\>}{\B[p_j,p_j]} p_j.
\]
Our CG method carefully constructs functions $p_0,p_1,\ldots,$ sequentially, such that $\B[p_i,p_j] = 0$ for $i\neq j$, in the hope that we may not need all of them to obtain a good approximation to $u$.

\subsection{The unpreconditioned CG method with a restricted right-hand side}\label{sec:RestrictedCG}
In order to tackle the first major issue highlighted in the introduction (see \cref{problem1}), we compose $\mL$ with a projection operator\footnote{The idea of composing a matrix with a projection operator to generate a Krylov subspace is also used for solving saddle-point problems~\cite{gould2014projected}.} to ensure that any solution from the constructed Krylov subspace satisfies the zero Dirichlet conditions of~\cref{eq:PDE}.

Let $\mV_0$ be an approximation space for the solution of~\cref{eq:PDE}. We wish to construct a projection onto $\mV_0$ and apply it after each operator-function product so that the constructed Krylov subspace is a subspace of $\mV_0$. We temporarily make the following assumptions:
\begin{description}
\item \begin{assumption}\label{A1} $\mV_0$ is a closed (potentially infinite-dimensional) subspace of $\Htwozero$, and
\end{assumption}
\item \begin{assumption} \label{A2} $f\in \mV_0 $. \end{assumption}
\end{description}
In \cref{sec:preconditioning}, we introduce a preconditioner that acts as a ``smoother" to eliminate the need for \cref{A1} and allows us to set $\mV_0 = \mathcal{H}^1_0(\Omega)$. We avoid~\cref{A2} by solving an ancillary problem (see \cref{sec:GeneralRHS}).

Proceeding under~\cref{A1,A2}, we define an orthogonal projection operator onto $\mV_0$ (because $\mV_0$ is a closed subspace of $\Ltwo$) as
\[
\Pi_{\mV_0} \phi = \argmin_{p\in\mV_{0}} \|\phi - p\|, \qquad \Pi_{\mV_0} : \Ltwo\rightarrow \mathcal{V}_{0}.
\]
We work with the modified operator $\smash{\Pi_{\mV_0}^*\mL\Pi_{\mV_0}}: \Ltwo \rightarrow \mV_{0}$, where $\smash{\Pi_{\mV_0}^*}: \Ltwo\rightarrow \mathcal{V}_{0}$ is the adjoint of $\Pi_{\mV_0}$ over the $L^2$ inner-product.  Since $\Pi_{\mV_0} $ is an orthogonal projection, it is self-adjoint, i.e., $ \Pi_{\mV_0}^* = \Pi_{\mV_0} $~\cite[Chap.~5]{rynne2000linear}. This is important as it implies that the range of $\smash{\PiLPi}$ is $\mV_0$, and that the operator $
\smash{\Pi_{\mV_0}^*\mL\Pi_{\mV_0}}$ is self-adjoint. Consequently, it is reasonable to imagine applying a CG method with $\smash{\Pi_{\mV_0}^*\mL\Pi_{\mV_0}}$.

The operator $\smash{\PiLPi}: \Ltwo \rightarrow \mV_{0}$ is well-defined since $\mV_0\subset \Htwozero$, and we are interested in Krylov subspaces of the form
\begin{equation}
\mathcal{K}_k(\Pi_{\mV_0}^* \mathcal{L}\Pi_{\mV_0},f) = {\rm Span}\!\left\{\! f,\PiLPi f,\ldots,(\PiLPi)^{k-1}f\!\right\},\qquad k\geq 1.
\label{eq:NaiveKrylov2}
\end{equation}
Since $f\in\mV_{0}$, we know that $\mathcal{K}_k(\smash{\Pi_{\mV_0}^* \mathcal{L}\Pi_{\mV_0}},f)\subseteq \mV_{0}$ so that the boundary conditions are successfully incorporated into the Krylov subspace. Therefore, any iterative method that constructs iterates from the Krylov subspace in~\cref{eq:NaiveKrylov2} automatically imposes zero Dirichlet boundary conditions.

An unpreconditioned CG method can now be derived that generates iterates $u_0 = 0, u_1,u_2,\ldots,$ such that
\[
u_{k} = \argmin_{v\in\mathcal{K}_k(\Pi_{\mV_0}^* \mathcal{L}\Pi_{\mV_0},f)}\| u - v\|_{\mathcal{L}},
\]
where $u$ is the exact solution to~\cref{eq:PDE}\footnote{ This follows from the fact that the discretization error is $\mathcal{B}$-orthogonal to the algebraic error in a Galerkin method~\cite[Thm.~2.5.2]{liesen2013krylov}.}. The derivation of this method follows almost immediately from the CG method for matrices~\cite[Alg.~38.1]{trefethen1997numerical}, where in the derivation terms of the form $x^TAy$ are replaced by $\B[\phi,\psi]$, $x^Ty$ by $\<\phi,\psi\>$, and $Ax$ by $\PiLPi \phi$. The resulting unpreconditioned CG method for~\cref{eq:PDE} is given in~\cref{alg:CGBVP}. We also give the matrix CG method in~\cref{alg:CGmatrix} for comparison, and we emphasize that the two algorithms are essentially the same except the operations~\cref{alg:CGmatrix} are with vectors and matrices while the operations in~\cref{alg:CGBVP} are with functions and operators.

\vspace{-.2cm}

\begin{flushleft}
\begin{minipage}[t]{.47\textwidth}
\null
\begin{algorithm}[H]
\caption{The CG method for solving $Ax=b$, where $A\in\mathbb{R}^{n\times n}$ is a symmetric positive definite matrix and $b\in\mathbb{R}^{n\times 1}$.\vspace{0.05cm}
}\label{alg:CGmatrix}
\begin{algorithmic}[1]
\STATE{ Set $x_0=0$, $r_0=b$, and $p_0=b$}
\FOR{$k\! = \!0, 1,\ldots$ (until converged)}
\STATE{$\alpha_k= r_k^Tr_k \big/ (p_k^TAp_k)$}
\STATE{$x_{k+1}=x_k+\alpha_k p_k$}
\STATE{$r_{k+1}=r_k-\alpha_kAp_k$}
\STATE{$\beta_k= r_{k+1}^Tr_{k+1} \big/  r_k^Tr_k$}
\STATE{$p_{k+1}=r_{k+1}+\beta_k p_k$}
\ENDFOR
\end{algorithmic}
\end{algorithm}
\end{minipage}
\hspace{0.5cm}
\begin{minipage}[t]{.47\textwidth}
  \vspace{0pt}
\begin{algorithm}[H]
\caption{The CG method for~\cref{eq:PDE}, where $\smash{\mathcal{L}}$ is self-adjoint with $a(x)>0$ and $c(x)\geq 0$, and $\smash{f\in\mV_{0}}$. \vspace{0.45cm}}\label{alg:CGBVP}
\begin{algorithmic}[1]
\STATE{ Set $u_0=0$, $r_0=f$, and $p_0=f$}
\FOR{$k\!=\!0,1,\ldots$ (until converged)}
\STATE{$\alpha_k=\< r_k, r_k \> \big/ \B[p_k,p_k]$}
\STATE{$u_{k+1}=u_k+\alpha_k p_k$}
\STATE{$r_{k+1}=r_k-\alpha_k\PiLPi p_k$}
\STATE{$\beta_k=\< r_{k+1},r_{k+1}\>\big/\< r_k,r_k\>$}
\STATE{$p_{k+1}=r_{k+1}+\beta_k p_k$}
\ENDFOR
\end{algorithmic}
\end{algorithm}
\end{minipage}
\end{flushleft}

\vspace{.5cm}

For~\cref{alg:CGBVP} to be well-defined we must check that: (1) $r_0,r_1,\ldots,$ are in $\Ltwo$ so that $\<r_k,r_k\>$ is valid, (2) $p_0,p_1,\ldots,$ are in $\Ltwo$ so that $\PiLPi p_k$ is well-defined, and (3) $p_0,p_1,\ldots,$ are in $\mathcal{H}^1_0(\Omega)$ so that $\mathcal{B}[p_k,p_k]$ is valid. All these statements hold when $f\in \mathcal{V}_0\subset \Htwozero$ and can be verified by mathematical induction.

The CG method in~\cref{alg:CGBVP} is theoretically justified for any $\mV_0$ that is a closed subspace of $\Htwozero$. In particular, this includes the space $\mV_0 = \{v\in \mathcal{P}_n : v(\pm 1) = 0\}$ for some integer $n$, where $\mathcal{P}_n$ is the space of polynomials of degree $\leq n$. Furthermore, if the basis for $\mathcal{P}_n$ is selected to be the Legendre polynomials, then the CG method in~\cref{alg:CGBVP} is closely related to applying the CG method to a Legendre--Galerkin discretization of~\cref{eq:PDE}~\cite[Sec.~4.1]{shen2011spectral}. The advantage of~\cref{alg:CGBVP} is that it provides important insights into how to derive a preconditioned CG method (see~\cref{sec:precondCG}).

The convergence of the unpreconditioned CG method in~\cref{alg:CGBVP} is generically poor (see~\cref{fig:unpreconditioned}). The unboundedness of the differential operator means that $k={\rm dim}(\mV_{0})$ iterations are typically necessary (see~\cref{fig:unpreconditioned}) and, in our setting, $\mV_0$ could potentially be an infinite dimensional subspace.

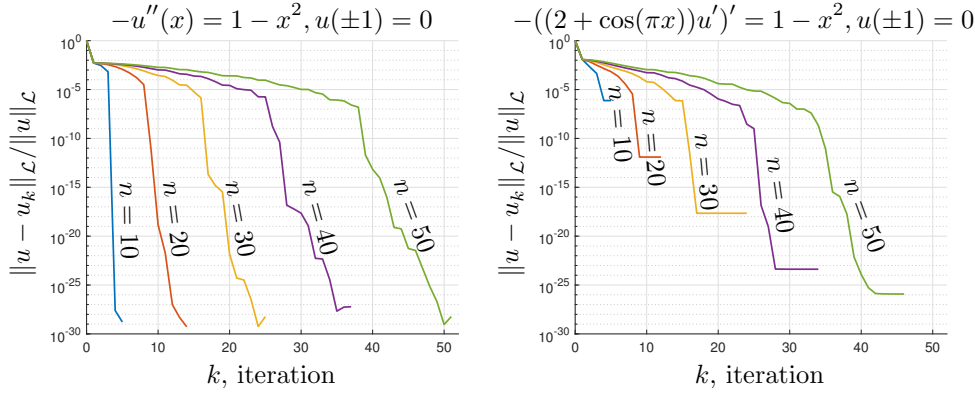
\begin{figure}
\centering
\begin{minipage}{.49\textwidth}
\begin{overpic}[width=\textwidth]{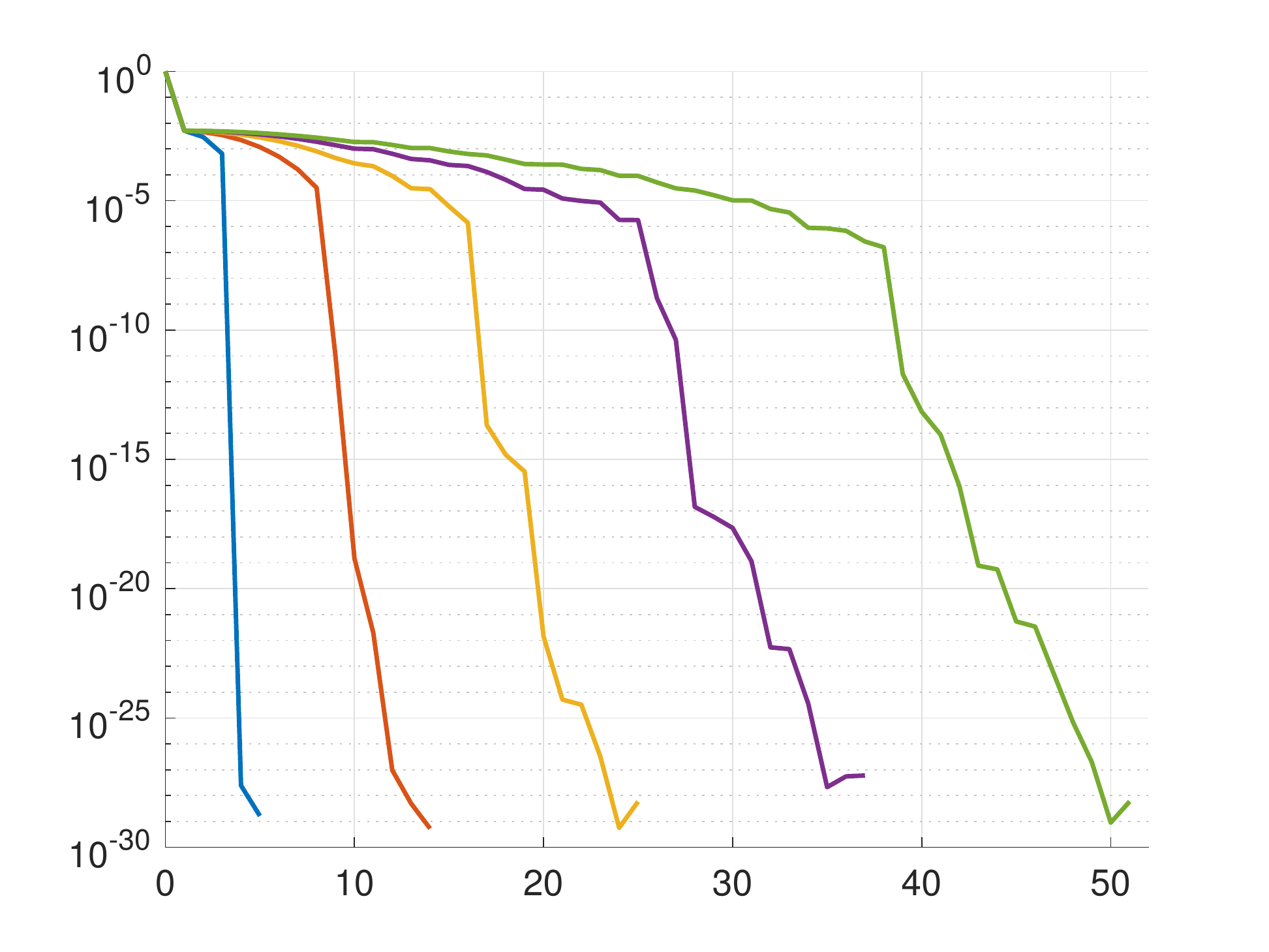}
\put(38,-2) {$k$, iteration}
\put(18,72) {$-u''(x) = 1-x^2, u(\pm 1)=0$}
\put(-3,20) {\rotatebox{90}{$\| u - u_k\|_{\mathcal{L}}/\|u\|_{\mathcal{L}}$}}
\put(20,40) {\rotatebox{-89}{$n = 10$}}
\put(29,40) {\rotatebox{-83}{$n = 20$}}
\put(42,40) {\rotatebox{-80}{$n = 30$}}
\put(58,40) {\rotatebox{-75}{$n = 40$}}
\put(77,40) {\rotatebox{-70}{$n = 50$}}
\end{overpic}
\end{minipage}
\begin{minipage}{.49\textwidth}
\begin{overpic}[width=\textwidth]{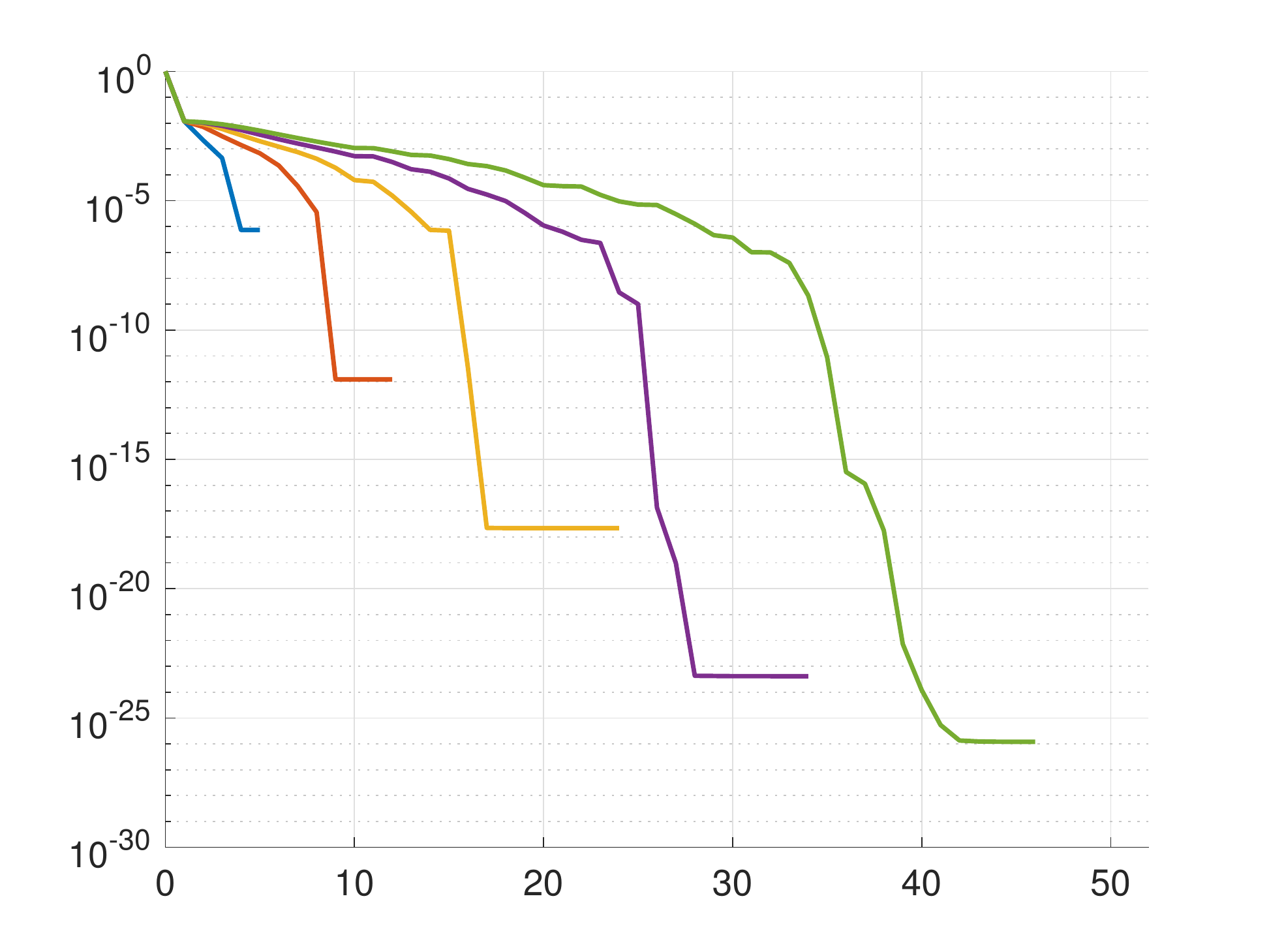}
\put(38,-2) {$k$, iteration}
\put(-3,20) {\rotatebox{90}{$\| u - u_k\|_{\mathcal{L}}/\|u\|_{\mathcal{L}}$}}
\put(0,72) {$-((2+\cos(\pi x))u')' = 1-x^2, u(\pm 1)=0$}
\put(20,60) {\rotatebox{-85}{$n = 10$}}
\put(27,55) {\rotatebox{-85}{$n = 20$}}
\put(38,50) {\rotatebox{-85}{$n = 30$}}
\put(53,45) {\rotatebox{-82}{$n = 40$}}
\put(69,40) {\rotatebox{-70}{$n = 50$}}
\end{overpic}
\end{minipage}
\caption{Convergence of the unpreconditioned CG method when $\mV_0 = \{v\in \mathcal{P}_n : v(\pm 1) = 0\}$ and $10\leq n\leq 50$, where $\mathcal{P}_n$ is the space of polynomials of degree $\leq n$. Left: The CG error when solving $\smash{-u'' = 1-x^2}$ on $(-1,1)$ and $u(\pm1)=0$.  Right: The CG error when solving $\smash{-((2+\cos(\pi x))u')' = 1-x^2}$ on  $(-1,1)$ and $u(\pm1)=0$. The unpreconditioned CG method here is rarely useful because differential operators are unbounded and the number of required CG iterations is generically ${\rm dim}(\mV_{0})$. To overcome this, we develop operator preconditioners (see \cref{sec:preconditioning}).}
\label{fig:unpreconditioned}
\end{figure}

\subsection{Operator preconditioning} \label{sec:preconditioning}
Improving the convergence of~\cref{alg:CGBVP} requires the development of preconditioners.
The preconditioned CG method for solving $Ax=b$ is equivalent to applying the CG method to $R^TARy = R^Tb$, where $x = Ry$ and $R$ is a square matrix~\cite[Sec.~8.1]{meurant2006lanczos}. Motivated by this, we consider solving
\begin{equation}
\mR^* \mathcal{L} \mR v  = \mR^*f\quad \text{ on }\Omega=(-1,1), \qquad (\mR v)(\pm 1) = 0,
\label{eq:preconditionedPDE}
\end{equation}
where $\mR:\Ltwo\rightarrow \Ltwo$ is a linear operator and $\mR^*$ is the adjoint of $\mR$, i.e., $\langle \mR^*\phi,\psi\rangle = \langle \phi,\mR \psi\rangle$ for all $\phi,\psi\in \Ltwo$. We call $\mR$ an {\em operator preconditioner}.\footnote{In the Petrov--Galerkin literature, the concept of ``operator preconditioning" is similar and refers to a recipe for constructing preconditioners so that they are robust with respect to the choice of trial and test basis~\cite{hiptmair2006operator}. A related concept is ``equivalent preconditioners", where one constructs a preconditioner by simplifying the given differential operator~\cite{axelsson2009equivalent}.}

We make the following requirements on the operator preconditioner $\mR:\Ltwo\rightarrow \Ltwo$, which appear to be necessary in our framework to overcome the remaining two major issues highlighted in the introduction (see \cref{problem2,problem3}):
\begin{description}[leftmargin=*,noitemsep]
\item[{\bf Bounded:}] The preconditioner $\mR:\Ltwo\rightarrow \Ltwo$ is a bounded linear operator. That is,  $\|\mR\|_{\text{op}} = \sup_{\phi\in \Ltwo, \|\phi\| = 1} \|\mR \phi\| < \infty$.
\item[{\bf Smoother:}] The preconditioner and its adjoint over $\Ltwo$ are smoothers, i.e., $\mR:\Ltwo \rightarrow \Hone$,  $\mR:\Hone \rightarrow \Htwo$, $\mR^*:\Ltwo \rightarrow \Hone$, and $\mR^*:\Hone \rightarrow \Htwo$.\footnote{Note that if $f \in L^2( \Omega)$ this implies that $\mR^*f \in \Hone$ and $\mR^* \mL \mR: \Hone \rightarrow \Hone $.}
\item[{\bf Preconditioner for the Laplacian:}] There are constants $0<\gamma_0\leq \gamma_1<\infty$ such that $\gamma_0\|\phi\|^2\leq \| (\mR \phi)'\|^2 \leq \gamma_1\|\phi\|^2$ for all $\phi\in \Ltwo$.
\end{description}
A natural operator preconditioner for~\cref{eq:PDE}, and our canonical choice, is the indefinite integration operator $\mR:\Ltwo \rightarrow \Ltwo$, defined as
\begin{equation}
(\mR \phi)(x) = \int_{-1}^x \phi(s) ds, \qquad (\mR^*\phi)(x) = \int_{x}^1 \phi(s) ds, \qquad \phi\in\Ltwo.
\label{eq:IntegralPreconditioner}
\end{equation}
The preconditioner and its adjoint act as ``smoothers" and $\|\mR\|_{\text{op}} = 4/\pi<\infty$~\cite[Prob.~188]{halmos2012hilbert}.  If $\mathcal{L}u = -u''$, then the associated bilinear form of the operator $\mR^*\mathcal{L}\mR$ is
\[
\mathcal{B}[\mR \phi,\mR \psi] = \!\!\int_{-1}^1 \!\!\left(\int_{-1}^x\!\! \phi(s) ds\right)'\!\! \left(\int_{-1}^x \!\!\psi(s) ds\right)' \!\!dx  =\!\! \int_{-1}^1 \!\!\phi(x)\psi(x) dx = \<\phi,\psi\>,
\]
where $\phi,\psi\in \Ltwo$. Therefore, $\mR$ is a preconditioner for the Laplacian with $\gamma_0 = \gamma_1 = 1$ so that the $\mathcal{R}$ in~\cref{eq:IntegralPreconditioner} satisfies all of our requirements.

The integral preconditioner in~\cref{eq:IntegralPreconditioner} appears throughout the literature and is exploited to construct preconditioners for finite element discretizations~\cite{kirby2010functional} as well as for spectral Galerkin discretizations~\cite[Chap.~4]{canuto2010spectral}.

\subsection{The preconditioned CG method}\label{sec:precondCG}
With an operator preconditioner in hand, we are able to derive a satisfying operator CG method. In order for $\Honezero$ to be the solution space for $u = \mR v$ in~\cref{eq:preconditionedPDE}, the space $\mW_0 = \{\phi\in \Ltwo : \mR \phi \in \Honezero\}$ must be the approximation space for $v$ in~\cref{eq:preconditionedPDE}. Moreover, instead of assuming that $f\in \mV_0$, we must now work under the the following assumption:
\begin{description}
\item \begin{assumption} \label{A2p} $\mR^*f\in \mW_0$. \end{assumption}
\end{description}
We no longer need~\cref{A1} and we remove~\cref{A2p} in \cref{sec:GeneralRHS}.
Since we are using a preconditioner and the approximation space for the solution of~\cref{eq:preconditionedPDE} is $\mW_0$, we first need to design an orthogonal projection operator $\Pi_{\mW_0} : \Ltwo \rightarrow \mW_0$. This task appears challenging for general preconditioners $\mathcal{R} $.
However, when $\mathcal{R}$ is the indefinite integral preconditioner in~\cref{eq:IntegralPreconditioner}, we note that $\mW_0$ is the space of $\Ltwo$ functions with zero mean. Moreover, $(\mathcal{R}\phi)(-1) = 0$ for all $\phi \in \Ltwo$, and hence we find that the orthogonal projection $\Pi_{\mW_0}:\Ltwo \to \mW_0$ is given by
\[
\Pi_{\mW_0}\phi = \phi - \frac{1}{2}\int_{-1}^1 \phi(s) ds.
\]
It is easy to verify that this projection operator is self-adjoint, and thus orthogonal:
\[
\<\Pi_{\mW_0}\phi,\psi\> = \<\phi,\psi\> - \frac{1}{2}\int_{-1}^1 \phi(s)ds\int_{-1}^1 \psi(s)ds = \<\phi, \Pi_{\mW_0}\psi\>, \qquad \phi,\psi\in \Ltwo.
\]

Given an orthogonal projection operator $\Pi_{\mW_0}:\Ltwo\to \mW_0$, we can derive a preconditioned CG method that constructs iterates $v_0 = 0, v_1,v_2,\ldots,$ so that $u_k = \mR v_k$ approximates the solution to~\cref{eq:PDE}. The Krylov subspace of interest is now
\begin{equation}
\mathcal{K}_k(\mathcal{T}, \mR^*f) = {\rm Span}\left\{\mR^*f, \mathcal{T}(\mR^*f),\ldots,\mathcal{T}^{k-1}\mR^*f\right\}, \qquad \mathcal{T} = \Pi_{\mathcal{W}_0}^*\mR^*\mathcal{L}\mR\Pi_{\mathcal{W}_0},
\label{eq:preconditionedKrylov}
\end{equation}
where $\mathcal{K}_k(\mathcal{T},\mR^*f)\subset \mathcal{W}_0$ because \cref{A2p} ensures that $\mR^*f\in \mW_0$.  The associated preconditioned CG method is given in~\cref{alg:CGBVP2}.

\begin{algorithm}[H]
\caption{The preconditioned CG method for~\cref{eq:PDE}, where $\smash{\mathcal{L}}$ is self-adjoint with $a(x)>0$ and $c(x)\geq 0$, and $\smash{\mR^*f\in\mW_{0}}$.}\label{alg:CGBVP2}
\begin{algorithmic}[1]
\STATE{ Set $v_0=0$, $r_0=\mR^*f$, and $p_0=\mR^*f$}
\FOR{$k =0,1,\ldots,$ (until converged)}
\STATE{$\alpha_k=\< r_k, r_k \> \big/ \B[\mR p_k,\mR p_k]$}
\STATE{$v_{k+1}=v_k+\alpha_k p_k$}
\STATE{$r_{k+1}=r_k-\alpha_k\mathcal{T} p_k$}
\STATE{$\beta_k=\< r_{k+1},r_{k+1}\>\big/\< r_k,r_k\>$}
\STATE{$p_{k+1}=r_{k+1}+\beta_k p_k$}
\STATE{$u_{k+1} = \mR v_{k+1}$}
\ENDFOR
\end{algorithmic}
\end{algorithm}
To verify that~\cref{alg:CGBVP2} is well-defined we check that: (1) $r_0,r_1,\ldots,$ are in $\Ltwo$ so that $\<r_k,r_k\>$ is valid, (2) $p_0,p_1,\ldots,$ are in $\Ltwo$ so that $\mathcal{T} p_k$ and $\mathcal{B}[\mR p_k,\mR p_k]$ are valid operations. All these statements hold when $\mR^*f\in \mathcal{W}_0$ and can be proved by mathematical induction.

The preconditioned CG method in~\cref{alg:CGBVP2} immediately inherits many of the theoretical properties from the CG method for matrices~\cite{meurant2006lanczos}. Here are two immediate facts that are analogous to familiar results for the matrix CG method:
\begin{lemma}
The functions $r_0,r_1,\ldots,$ in~\cref{alg:CGBVP2} satisfy $\<r_i,r_j\>=0$ for $i\neq j$. Moreover, the functions $p_0,p_1,\ldots,$ satisfy $\B[\mR p_i,\mR p_j]=0$ for $i\neq j$.
\label{lem:orthogonality}
\end{lemma}
\begin{proof}
The constant $\alpha_k$ is selected so that $\<r_{k+1},r_k\>=0$ for $k\geq 0$. This gives the formula $\alpha_k = \<r_k,r_k\>/\B[\mR r_k,\mR p_k]$, which can be simplified to the formula in~\cref{alg:CGBVP2} since $r_{k+1} = p_{k+1} -  \beta_{k}p_k$.  The constant $\beta_k$ is selected so that $\B[\mR p_{k+1},\mR p_k]=0$ for $k\geq 0$. This gives the formula $\beta_k = -\B[\mR r_{k+1},\mR p_k]/\B[\mR p_k,\mR p_k]$, which can be simplified to the formula in~\cref{alg:CGBVP2} since $r_{k+1} = r_k - \alpha_k \mathcal{T}p_k$. The result immediately follows.
\end{proof}

\Cref{lem:orthogonality} also shows that~\cref{alg:CGBVP2} is solving a best approximation problem.
\begin{theorem}
Let $u_0=0,u_1,\ldots,$ be the CG iterates from~\cref{alg:CGBVP2} and $u$ the solution to~\cref{eq:PDE}. Then,
\[
u_k = \argmin_{p\in\mathcal{X}_k} \| u-p\|_{\mathcal{L}}, \qquad k\geq 1,
\]
where $\mathcal{X}_k = \{p\in\Honezero : p = \mR q, q\in \K_k(\mathcal{T},\mR^*f) \}$.
\label{thm:best}
\end{theorem}
\begin{proof}
From~\cref{lem:orthogonality}, we find that $\B[\mR (v-v_k),\mR p_j]=0$ for $j\geq k+1$, where $u = \mR v$. In other words, we have
\[
v_k = \argmin_{q\in\K_k(\mathcal{T},\mR^*f)} \| v-q\|_{\mathcal{T}}.
\]
Since $\| v-q\|_{\mathcal{T}}^2 = \B[\mR (v-q), \mR (v-q)] = \B[u -\mR q, u -\mR q] = \|u-\mR q\|_\mathcal{L}^2$, this is equivalent to $v_k = \argmin_{q\in\K_k(\mathcal{T},\mR^*f)} \|u-\mR q\|_\mathcal{L}$. Finally, we note that $u_k = \mR v_k$ and therefore, $u_k = \argmin_{p\in\mathcal{X}_k} \| u-p\|_{\mathcal{L}}$, where $\mathcal{X}_k = \{p\in\Honezero : p = \mR q, q\in \K_k(\mathcal{T},\mR^*f) \}$.
\end{proof}

\Cref{thm:best} shows that~\cref{alg:CGBVP2} is calculating the best approximation from $\mathcal{X}_k$ to $u$ in the $\|\cdot \|_\mathcal{L}$ norm and also guarantees that the error $e_k = \| u - u_k\|_{\mathcal{L}}$ is monotonically non-increasing, i.e.,
\[
\| u - u_{k+1}\|_{\mathcal{L}}\leq \| u - u_k\|_{\mathcal{L}}, \qquad k\geq 0.
\]

In practice, designing good preconditioners is paramount for an efficient BVP solver. One could imagine being confronted with the same dilemma as preconditioning the CG method for matrices. On the one hand, we want to select $\mR$ so that $\mR \phi$ can be computed efficiently for any $\phi\in\mathcal{W}_0$. On the other hand, we want $\mathcal{T}$ to be a well-conditioned operator over $\mathcal{W}_0$ (see~\cref{eq:ConditionNumber}). Here, we have an additional desire that is not present for matrices: we would like an efficient algorithm to compute $\Pi_{\mW_0}\psi$ for any $\psi\in \Ltwo$, where $\Pi_{\mW_0}:\Ltwo\to \mW_0$ is an orthogonal projection operator (see \cref{sec:CG1dAlgorithm}).  In this paper, we always select $\mathcal{R}$ to be the indefinite integral operator in~\cref{eq:IntegralPreconditioner}.

\subsection{Convergence theory for the preconditioned CG method}\label{sec:convergence}
In this section, we show that the preconditioned CG method converges at a geometric rate when the operator preconditioner is bounded, is a smoother, and is a preconditioner for the Laplacian (see \cref{sec:preconditioning}). The standard bound on the convergence of the CG method for $Ax=b$ involves the condition number of $A$~\cite[Chap.~2]{meurant2006lanczos}. Though this bound is not always descriptive, it is explicit and is the first canonical convergence result.  Similarly, the convergence of our operator CG method can be bounded using the condition number of the operator $\mR^*\mathcal{L}\mR$ from a restricted subspace of $\Ltwo$.

The condition number of $\mR^*\mathcal{L}\mR:\mW_0\to \Ltwo$ is given by~\cite{kirby2010functional}:
\begin{equation}
\kappa_{\mW_0}(\mR^*\mathcal{L}\mR) = \frac{\sup_{\phi\in \mW_0, \|\phi\|=1} \B[\mR \phi,\mR \phi]}{\inf_{\phi\in \mW_0, \|\phi\|=1} \B[\mR \phi,\mR \phi]},
\label{eq:ConditionNumber}
\end{equation}
where $\mW_0 = \left\{\phi\in\Ltwo: \mR\phi\in\Honezero\right\}$.
The following theorem bounds $\kappa_{\mW_0}(\mR^*\mathcal{L}\mR)$ and is used in~\cref{cor:CGconvergence} to derive a CG convergence bound.
\begin{theorem}
Let $\Omega=(-1,1)$, $a, c \in L^\infty(\Omega)$, $a(x)>0$ for $x\in\Omega$, $c(x)\geq 0$ for $x\in\Omega$, and $\mathcal{L}u = -(a(x)u'(x))' + c(x)u$ with bilinear form $\mathcal{B}:\Honezero\times \Honezero \rightarrow \mathbb{R}$. Given an operator preconditioner $\mR$ that is bounded, is a smoother, and is a preconditioner for the Laplacian (see \cref{sec:preconditioning}), the (restricted) condition number of $\mR^*\mathcal{L}\mR$ is bounded. Furthermore,
\[
\kappa_{\mW_0}(\mR^*\mathcal{L}\mR) \leq \frac{\gamma_1\|a\|_\infty + \|c\|_\infty \|\mR\|_{\text{op}}^2}{\gamma_0\inf_{x\in\Omega} |a(x)|},
\]
where $\mW_0 = \{\phi\in \Ltwo : \mR \phi \in \Honezero\}$ and $\|\mR\|_{\text{op}} = \sup_{\phi\in \mW_0, \|\phi\|=1} \|\mR \phi\|$.
\label{thm:ConditionNumber}
\end{theorem}
\begin{proof}
If $\phi\in\mW_0$, then $\mR \phi \in \mathcal{H}^1_0(\Omega)$ and we have
\begin{equation}
\mathcal{B}[\mR \phi,\mR \phi] = \int_{-1}^1 a(x) (\mR \phi)'(x) (\mR \phi)'(x)dx + \int_{-1}^1 c(x)(\mR \phi(x))^2 dx.
\label{eq:SplitInequality}
\end{equation}
The first term in~\cref{eq:SplitInequality} can be bounded as follows:
\[
\int_{-1}^1 a(x) (\mR \phi)'(x) (\mR \phi)'(x)dx \leq \|a\|_\infty \| (\mR \phi)'\|^2 \leq \gamma_1 \|a\|_\infty \|\phi\|^2,
\]
where the last inequality uses the fact that $\mR$ is a preconditioner for the Laplacian (see \cref{sec:preconditioning}). We also find that $\int_{-1}^1 a(x) (\mR \phi)'(x) (\mR \phi)'(x)dx \geq \gamma_0\inf_{x\in\Omega} |a(x)| \|\phi\|^2$. For the second term in~\cref{eq:SplitInequality}, we simply have
\[
0 \leq \int_{-1}^1 c(x)(\mR \phi(x))^2 dx \leq \|c\|_\infty \|\mR \phi\|^2 \leq \|c\|_\infty \|\mR\|_{\text{op}}^2 \|\phi\|^2.
\]
The bound on $\kappa_{\mW_0}(\mR^*\mathcal{L}\mR)$ immediately follows.
\end{proof}

Similar statements to~\cref{thm:ConditionNumber} appear in the literature on operator preconditioners for Galerkin discretizations~\cite{hiptmair2006operator,kirby2010functional}. \Cref{thm:ConditionNumber} has a slightly different flavor because $\mR$ and $\mathcal{L}$ are operators.

In~\cref{eq:preconditionedKrylov}, the Krylov space is based on the operator $\mathcal{T} = \Pi_{\mathcal{W}_0}^*\mR^*\mathcal{L}\mR\Pi_{\mathcal{W}_0}$ and the (restricted) condition number of $\mathcal{T}$ immediately follows from~\cref{thm:ConditionNumber}.
\begin{corollary}
With the same assumptions as~\cref{thm:ConditionNumber}, we have
\[
\kappa_{\mathcal{W}_0}(\mathcal{T}) = \kappa_{\mW_0}(\mR^*\mathcal{L}\mR),\qquad \mathcal{T} = \Pi_{\mathcal{W}_0}^*\mR^*\mathcal{L}\mR\Pi_{\mathcal{W}_0},
\]
where $\Pi_{\mathcal{W}_0}: \Ltwo\rightarrow \mathcal{W}_0$ is the orthogonal projection operator onto $\mathcal{W}_0$.
\label{cor:RestrictedConditionNumber}
\end{corollary}

The bound on $\kappa_{\mathcal{W}_0}(\mathcal{T})$ allows us to prove that $\|u-u_k\|_{\mathcal{L}}$ geometrically decays to zero as $k\rightarrow\infty$.
\begin{corollary}
With the same assumptions as~\cref{thm:ConditionNumber}, let $u_0=0, u_1,\ldots,$ be the CG iterates from~\cref{alg:CGBVP2}. Then,
\begin{equation}
\|u-u_k\|_\mathcal{L}  \leq 2 \left( \frac{\sqrt{\kappa_{\mathcal{W}_0}(\mathcal{T})}-1}{\sqrt{\kappa_{\mathcal{W}_0}(\mathcal{T})}+1} \right)^k \| u \|_\mathcal{L}, , \qquad k\geq 0,
\label{eq:finalErrorBound}
\end{equation}
where $\mathcal{T} = \Pi_{\mathcal{W}_0}^*\mR^*\mathcal{L}\mR\Pi_{\mathcal{W}_0}$ and $u$ is the exact solution to~\cref{eq:PDE}.
\label{cor:CGconvergence}
\end{corollary}
\begin{proof}
\Cref{cor:RestrictedConditionNumber} shows that $\kappa_{\mathcal{W}_0}(\mathcal{T})$ is bounded. By copying the proof of the convergence bound for the CG method for matrices~\cite{meurant2006lanczos}, we find that the iterates $v_0=0,v_1,v_2,\ldots,$ satisfy
\[
\| v - v_k \|_{\mathcal{T} }  \leq 2 \left( \frac{\sqrt{\kappa_{\mathcal{W}_0}(\mathcal{T})}-1}{\sqrt{\kappa_{\mathcal{W}_0}(\mathcal{T})}+1} \right)^k \| v \|_{\mathcal{T}}, \qquad k\geq 0,
\]
where $u = \mR  v$.  The result follows since $\|v\|_\mathcal{T}^2 = \mathcal{B}[ \mR v, \mR v] = \|\mR v\|_\mathcal{L}^2$, and $u_k = \mR v_k$.
\end{proof}
\Cref{cor:CGconvergence} implies that the preconditioned CG method in~\cref{alg:CGBVP2} constructs iterates $u_0=0,u_1,u_2,\ldots,$ that converge geometrically to $u$ in the $\|\cdot \|_\mathcal{L}$ norm. In other words, for an accuracy goal of $0<\epsilon<1$ we require
\[
k\geq \Bigg\lceil \frac{\log{(2/\epsilon)}}{\log\left(\sqrt{\kappa_{\mathcal{W}_0}(\mathcal{T})}+1\right) - \log\left(\sqrt{\kappa_{\mathcal{W}_0}(\mathcal{T})}-1\right)} \Bigg\rceil,
\]
iterations to guarantee that $\| u - u_k \|_{\mathcal{L} }\leq \epsilon \|u\|_{\mathcal{L}}$. Here, $\lceil x \rceil$ denotes the smallest integer greater than or equal to $x$.  Since $\kappa_{\mathcal{W}_0}(\mathcal{T})$ is bounded, the preconditioned CG method can be terminated after a finite number of iterations.

\Cref{fig:preconditioned} shows the convergence of the preconditioned CG method compared to the error bound in~\cref{eq:finalErrorBound} when solving three BVPs using the indefinite integration preconditioner $\mR v = \int_{-1}^x v(s)ds$. The convergence behavior of the preconditioned CG method comes with theoretical guarantees, and is a vast improvement over the convergence of the unpreconditioned CG method (see~\cref{fig:unpreconditioned} (right)).

\begin{figure}
\centering
\begin{minipage}{.49\textwidth}
\begin{overpic}[width=\textwidth]{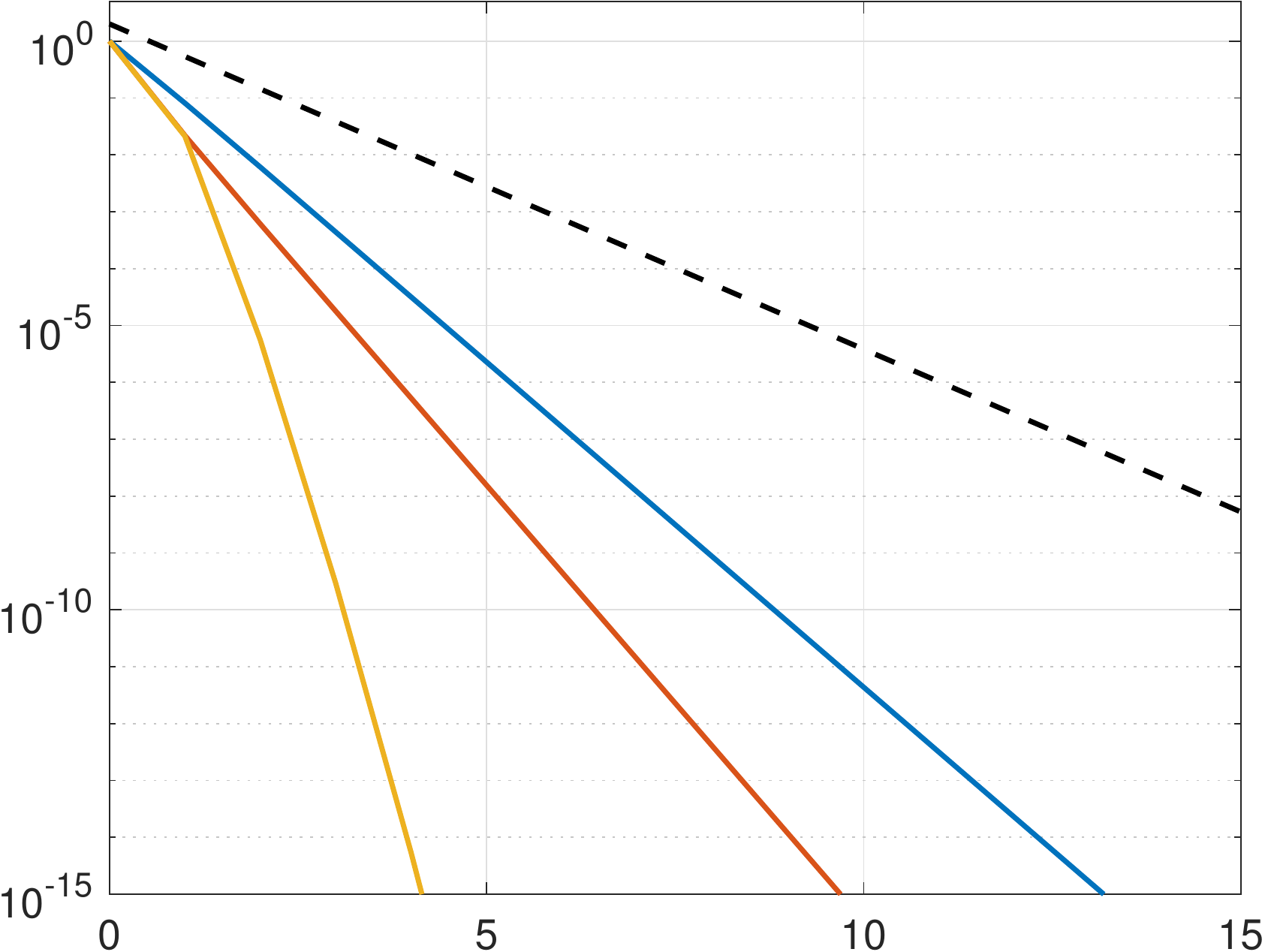}
\put(40,-2) {$k$, iteration}
\put(-5,20) {\rotatebox{90}{$\| u - u_k\|_{\mathcal{L}}/\|u\|_{\mathcal{L}}$}}
\put(50,60) {\rotatebox{-23}{$2\!\left(\frac{\sqrt{3}-1}{\sqrt{3}+1}\right)\!{}^{k}$}}
\put(50,40) {\rotatebox{-40}{(E1)}}
\put(40,36) {\rotatebox{-50}{(E2)}}
\put(25,35) {\rotatebox{-70}{(E3)}}
\end{overpic}
\end{minipage}
\caption{Convergence of the preconditioned CG method for three BVPs with zero Dirichlet boundary conditions. (E1):  $\smash{-((2+ \cos(\pi x))u')' = f}$ (blue line), (E2): $\smash{-( (1 + x^2)u' )' + (\tfrac{\pi}{4}\cos(\pi x))^2 u= f}$ (red line), and (E3): $\smash{-u'' + 2(\tfrac{\pi}{4})^2u = f}$ (yellow line) with $\smash{f = (1+x^2)^{-1}}$. \Cref{cor:CGconvergence} gives the same bound for these three examples (black dashed line). Note that $\mR^* f\not\in \mW_0$ so an ancillary problem is solved before applying the CG method for these three BVPs (see \cref{sec:GeneralRHS}).}
\label{fig:preconditioned}
\end{figure}
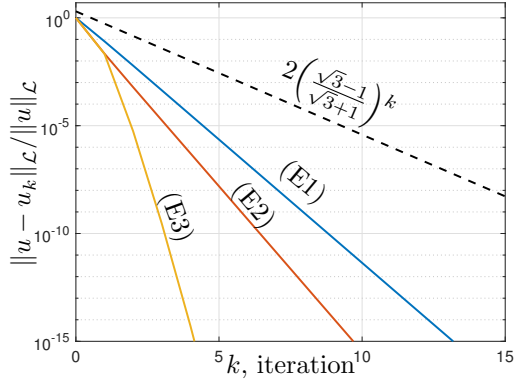

\subsection{General right-hand sides}\label{sec:GeneralRHS}
Here, we remove the assumption that $\mR^*f \in \mW_{0}$ by solving an ancillary problem that converts $f$ into a right-hand side that is amenable to our operator CG method.\footnote{In the standard Galerkin framework, one seeks to find a solution to~\cref{eq:PDE} via the weak formulation $\mathcal{B}[u,\psi] = \<f,\psi\>$ for all $\psi\in\Honezero$, even when $f\not\in \Honezero$. This is theoretically justified because $\Honezero$ is a dense subspace of $\Ltwo$. In our setup, the test space $\mW_0$ is not a dense subset of $L^2(\Omega)$ so solving an ancillary problem is necessary. }

Write the solution to~\cref{eq:preconditionedPDE} as $v = v_1 + v_2 $, where $v_2$ is any solution from $\mW_0$ that solves the following ancillary problem:
\begin{equation}
\left[\mR \mR^* \mL \mR v_2\right]\!(\pm 1) = \left[\mR \mR^*f \right]\!(\pm 1).
 \label{eq:BCequations}
\end{equation}
The remaining part of the solution, i.e., $v_1$, then satisfies
\[
\mR^* \mL \mR v_1 = \mR^*g, \qquad g = (f - \mR^* \mL \mR v_2).
\]
By construction, $\left[\mR\mR^*g\right]\!(\pm 1) = 0$ and since $g\in\Ltwo$, we find that $\mR^*g\in \mW_0$. Therefore, we can solve
\[
\mR^* \mL \mR v_1 = \mR^*g
\]
via the preconditioned CG method in~\cref{alg:CGBVP2}.

When $\mathcal{R}$ is the indefinite integral preconditioner in~\cref{eq:IntegralPreconditioner}, the condition at $-1$ in~\cref{eq:BCequations} is trivially satisfied and the ancillary problem reduces to solving
\begin{equation}
\int_{-1}^1 \!\!a(s) v_2(s) ds + \int_{-1}^1 \!\!\left(\int_{s}^1c(t)(t+1)dt\right) v_2(s) ds = \int_{-1}^1 (s+1)f(s) ds, \quad v_2\in \mW_0.
\label{eq:ancillaryProblem2}
\end{equation}
This problem can be solved efficiently by picking any $w \in \mW_0$ such that the lefthand side of~\cref{eq:ancillaryProblem2}, with $v_2$ replaced by $w$, is a scalar $\eta\neq 0$ and setting $v_2 = (\tfrac{1}{\eta}\smash{\int_{-1}^1} (s+1)f(s) ds)w$. Usually, $w(s)= s$ is an adequate choice.

\section{Practical realizations of the operator CG method}\label{sec:CG1dAlgorithm}
We now describe two realizations of the theoretical framework in \cref{sec:CGtheory} for solving~\cref{eq:PDE}.  While the theory in \cref{sec:CGtheory} works for the solution space $\Honezero$, in practice, we usually first define a dense subspace $\mV$ of $\Hone$ and associated subspace $\mW =\{ w \in L^2(\Omega) : \mR w \in \mV \}$ on which the operations performed by the CG method can be efficiently computed.  Provided that the operations performed by the CG method map functions from $\mW$ to $\mW$ and the right-hand side of~\cref{eq:PDE} and its variable coefficients are in $\mW$, the preconditioned CG method in \cref{sec:CGtheory} is unaware of the subspace $\mV$. In this section, we consider: (1) $\mV$ being the space of analytic functions and (2) $\mV$ being the space of continuous piecewise analytic functions (with a finite number of fixed breakpoints).  In these two cases the approximation space for the solution to~\cref{eq:PDE} is $\mV_0 = \{ \phi\in \mV : \phi(\pm1)=0\}\subset \Honezero$.

We have implemented (1) and (2) in Chebfun~\cite{Chebfun} in the \texttt{pcg} command, which follows the syntax of the standard MATLAB \texttt{pcg} command for matrices. Fortunately, object-oriented programming allows us to only have one implementation of the operator CG method for (1) and (2) as Chebfun automatically calls the appropriate underlying algorithms to compute inner-products, integrals, and derivatives via operator overloading.  This is one of the advantages of developing a Krylov-based solver that works independently from the underlying discretization of the solution and right-hand side.  Unlike most BVP solvers, our Krylov-based solvers have no fixed discretization. Instead, we let Chebfun automatically resolve the functions that appear during the operator CG method to machine precision~\cite{aurentz2017chopping}. A summary of the main operations that the preconditioned CG method requires is given in~\cref{tab:CGcosts}, along with the corresponding Chebfun commands.
\begin{table}
  \caption{Summary of the main operations that are required by the preconditioned CG method. The Chebfun commands that execute these mathematical operations are also given. Objected-oriented programming and operator overloading allows the same Chebfun command to employ different underlying algorithms depending on whether \texttt{p} and \texttt{q} are analytic or piecewise analytic.}
  \label{tab:CGcosts}
  \centering
 \begin{tabular}{ c c c}
 \toprule
Operation & Mathematical operation & Chebfun command\\[3pt]
 \midrule
 Preconditioner & $\int_{-1}^x p(s) ds$, $\int_{x}^1 p(s)ds$ & \texttt{cumsum(p)}, \texttt{sum(p)-cumsum(p)} \\[3pt]
 Differentiation & $p'(x)$ & \texttt{diff(p)} \\[3pt]
 Product & $p(x)q(x)$ & \texttt{p*q} \\[3pt]
 Inner-product & $\int_{-1}^1 p(s)q(s)ds$ & \texttt{p'*q} \\[3pt]
 Projector & $p - \frac{1}{2}\int_{-1}^1 p(s)ds$ & \texttt{p - mean(p)}\\[3pt]
\bottomrule
\end{tabular}
\end{table}

\subsection{Analytic functions}\label{sec:analytic}
Let $\mV$ be the space of functions that are analytic in an open neighborhood of $[-1,1]$ and consider the preconditioner $\mR \phi = \int_{-1}^x \phi(s) ds$. Note that the associated space $\mW$ is closed under indefinite integration, differentiation, and function product, and that $\mR $ is bounded, is a smoother, and preconditions the Laplacian. The choice of $\mV$ and $\mR$ completely determine a realization of the preconditioned CG method with the approximation space for the solution $\mV_0 = \{ \phi\in \mV : \phi(\pm1)=0\}$. Here, we are implicitly assuming that the variable coefficients in~\cref{eq:PDE} are analytic functions or have been approximated by analytic functions.

In order to implement an efficient practical algorithm, we approximate analytic functions to within machine precision $\epsilon_{\text{mach}}$ by Chebyshev expansions. That is, for some integer $n\geq 0$ that is adaptively determined~\cite{aurentz2017chopping}, we approximate an analytic function $\phi\in \mV$ by
\begin{equation}
\phi(x) \approx p(x) = \sum_{k=0}^n \alpha_k T_k(x),\qquad \|\phi-p\|_\infty < \epsilon_{\text{mach}}\|\phi\|_\infty,
\label{eq:ChebyshevExpansion}
\end{equation}
where $T_k(x)$ is the degree $k$ Chebyshev polynomial and $\|\,\cdot\,\|_\infty$ is the absolute maximum norm on $[-1,1]$. If $p$ is the Chebyshev interpolant of an analytic function $\phi$, then the Chebyshev expansion coefficients in~\cref{eq:ChebyshevExpansion} converge geometrically to zero~\cite[Chap.~8]{trefethen2013approximation}. Moreover, the expansion coefficients $\{\alpha_k\}$ in~\cref{eq:ChebyshevExpansion} can be computed in $\mathcal{O}(n\log n)$ via the discrete Chebyshev transform~\cite{gentleman1972implementing}.
To automatically resolve a function $\phi\in \mV$ to machine precision, we call the Chebfun command \texttt{p = chebfun(phi)}.

There are a number of operations that the CG method must perform on the adaptively determined Chebyshev expansions:
\begin{description}[leftmargin=*,noitemsep]

\item[{\bf Applying the preconditioner and its adjoint:}] For $p(x) = \sum_{k=0}^n \alpha_k T_k(x)$, we need to compute $\mR p = \int_{-1}^x p(s)ds$. The Chebyshev expansion coefficients for $\mathcal{R}p$ can be computed by using a simple recurrence relation~\cite[Sec.~8.1]{mason2002chebyshev}, costing $\mathcal{O}(n)$ operations. This is implemented in the Chebfun command \texttt{cumsum(p)}. Similarly, $\mR^*p$ can be computed with the Chebfun command \texttt{sum(p)-cumsum(p)} in $\mathcal{O}(n)$ operations.

\item[{\bf Applying the differential operator:}] For $p(x) = \sum_{k=0}^n \alpha_k T_k(x)$, we need to compute $\mathcal{L}p$.  If $\mathcal{L}p = -(a(x)p'(x))' + c(x)p(x)$ and $a(x)$ and $c(x)$ are analytic functions and represented by adaptively determined Chebyshev expansions, then we can compute $\mathcal{L}$ via the Chebun commands \texttt{Lp = -diff(a*diff(p))+c*p}. Computing the Chebyshev expansions of $p'(x)$ can be computed in $\mathcal{O}(n)$ operations via a recurrence relation~\cite[p.~34]{mason2002chebyshev} and the coefficients for $a(x)p(x)$ can be computed in $\mathcal{O}(N\log N)$ operations with a discrete Chebyshev transform~\cite{gentleman1972implementing}. Here, $N$ is the maximum polynomial degree required to resolve $a$ and $p$.

\item[{\bf Inner-products:}] Given $p(x) = \sum_{k=0}^n \alpha_k T_k(x)$ and $q(x) = \sum_{k=0}^n \beta_k T_k(x)$, we need to be able to compute
\[
\<p,q\> = \int_{-1}^1 p(s) q(s) ds.
\]
We compute this by Clenshaw--Curtis quadrature~\cite[Chap.~19]{trefethen2013approximation}, costing $\mathcal{O}(n\log n)$ operations.  The integral is computed by the Chebfun command \texttt{p'*q}.

\item[{\bf Applying the projection operator:}] For $p(x) = \sum_{k=0}^n \alpha_k T_k(x)$, we need to compute the projection
\[
\Pi_{\mW_0} p = p- \frac{1}{2}\int_{-1}^1p(s)ds.
\]
This can be achieved in $\mathcal{O}(n\log n)$ operations by using Clenshaw--Curtis quadrature for definite integration~\cite[Chap.~19]{trefethen2013approximation}. The projection operator is computed by Chebfun with the command \texttt{p-mean(p)}.

\end{description}

Since this realization of the preconditioned CG method employs adaptively selected polynomials to resolve the solution of~\cref{eq:PDE}, we compare our preconditioned CG method against adaptive implementations of the spectral collocation method\footnote{More precisely, we compare against rectangular spectral collocation~\cite{driscoll2015rectangular}, which performs a projection of the range of the matrices to automatically deal with boundary conditions of BVPs. Rectangular spectral collocation is employed by default in the \texttt{chebop} class of Chebfun~\cite{driscoll2008chebop}.} and the ultraspherical discretization~\cite{olver2013fast}. Both these adaptive spectral methods are implemented in Chebfun.

To do the comparison, we consider the family of BVPs parametrized by $\omega_1$ and $\omega_2$ such that
\[
-((2+\cos(\omega_1 \pi x))u'(x))' = f(x) \text{ on } \Omega = (-1,1), \qquad u(\pm1) = 0,
\]
where the right-hand side $f(x)$ is chosen so that $u(x) = \sin(\omega_2 \pi x)$ is the exact solution. We investigate two regimes: (a) $\omega_1$ fixed, $\omega_2\rightarrow \infty$ and (b) $\omega_2$ fixed, $\omega_1\rightarrow \infty$. In the first regime, a high degree polynomial is required to resolve the solution to machine precision while the variable coefficients of the BVP can be resolved by a low degree polynomial. This is a setting in which the ultraspherical spectral method is competitive with the preconditioned CG method (see~\cref{fig:CGtimings} (left)). In the second regime, the variable coefficients of the BVP require high degree polynomials to resolve, leading to dense spectral discretization matrices for both spectral collocation and the ultraspherical spectral method. In this setting, we find that it is computationally beneficial to employ our preconditioned CG method.

\begin{figure}
\centering
\begin{minipage}{.49\textwidth}
\begin{overpic}[width=\textwidth]{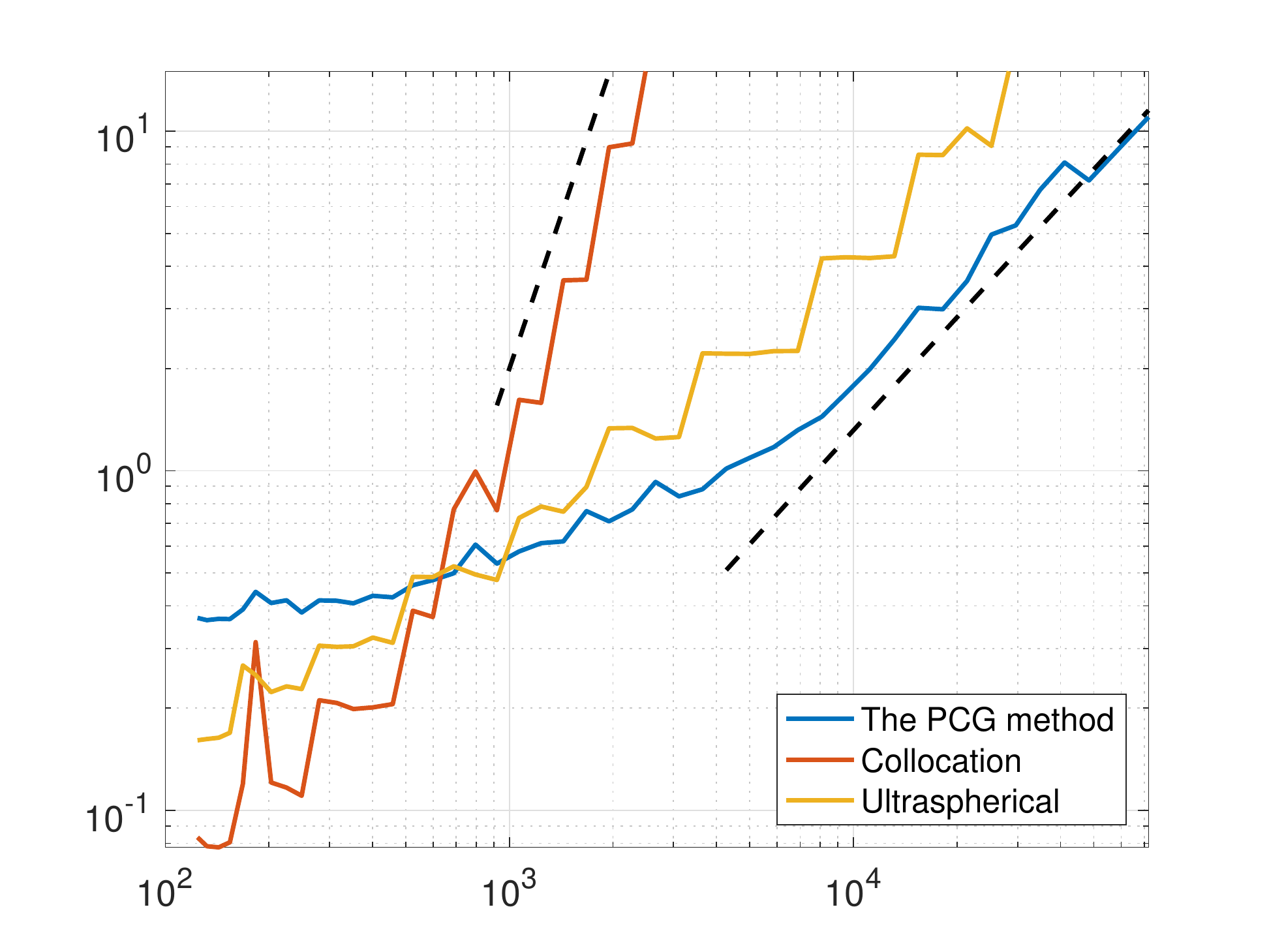}
\put(50,0) {$\omega_2$}
\put(0,20) {\rotatebox{90}{Execution time}}
\put(35,50) {\rotatebox{70}{\footnotesize{$\mathcal{O}(\omega_2^3)$}}}
\put(70,40) {\rotatebox{47}{\footnotesize{$\mathcal{O}(\omega_2\log \omega_2)$}}}
\end{overpic}
\end{minipage}
\begin{minipage}{.49\textwidth}
\begin{overpic}[width=\textwidth]{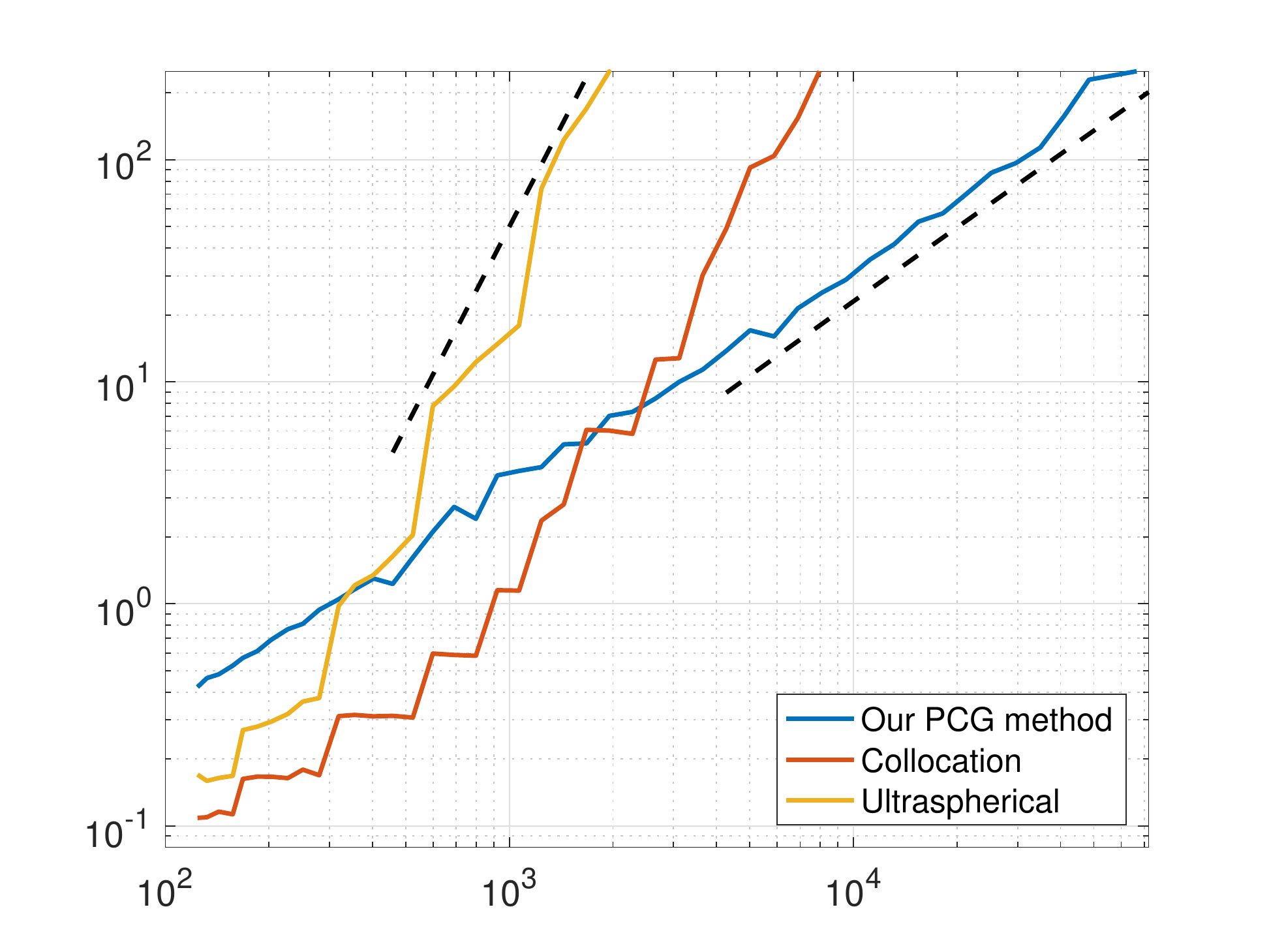}
\put(50,0) {$\omega_1$}
\put(0,20) {\rotatebox{90}{Execution time}}
\put(65,45) {\rotatebox{36}{\footnotesize{$\mathcal{O}(\omega_1\log \omega_1)$}}}
\put(30,50) {\rotatebox{68}{\footnotesize{$\mathcal{O}(\omega_1^3)$}}}
\end{overpic}
\end{minipage}
\caption{Comparison of execution timings of our preconditioned CG method (blue line), spectral collocation (red line), and the ultraspherical spectral method (yellow line) for the BVP $-((2+\cos(\omega_1 \pi x))u')' = f(x)$, $u(\pm 1)=0$, where $f$ is chosen so that $u(x) = \sin(\omega_2 \pi x)$ is the solution. All spectral methods are implemented in an adaptive manner to automatically resolve the BVP solution to essentially machine precision. Spectral collocation and the ultraspherical spectral method discretize the BVP and then solve the resulting linear system. Left: The parameter $\omega_2$ is increased while $\omega_1 = 10$, which defines a family of BVPs for which the solution requires a high polynomial degree to resolve to machine precision. Right: The parameter $\omega_1$ is increased while $\omega_2 = 10$, which defines a family of BVPs for which the variable coefficients require a high polynomial degree to resolve to machine precision. The polynomial degree required to resolve $\cos(\omega_1 \pi x)$ and $\sin(\omega_2 \pi x)$ on $[-1,1]$ to machine precision is $\mathcal{O}(\omega_1)$ and $\mathcal{O}(\omega_2)$, respectively.}.
\label{fig:CGtimings}
\end{figure}
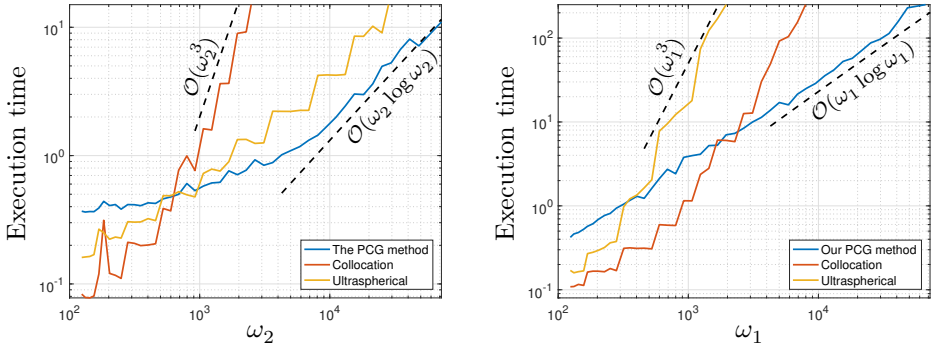

From these experiments and others, we learn that the preconditioned CG method is computationally beneficial compared to standard spectral methods employing direct solvers when spectral methods generate linear systems that are large and dense.  A similar comparison can be made between direct and iterative solvers for linear systems.

\subsection{Continuous functions that are piecewise analytic}
Let $\mV\subset \Hone$ be the space of continuous functions that are piecewise analytic with a finite number of fixed breakpoints $-1 = x_0 < x_1 < \cdots < x_{M+1} = 1$. That is, the space of continuous functions $\phi$ such that $\smash{\phi|_{[x_i,x_{i+1}]}}$ is analytic in a neighborhood of $[x_i,x_{i+1}]$ for $0\leq i\leq M$. Again, we take the preconditioner to be $\mR \phi = \int_{-1}^x \phi(s) ds$. The induced space $\mW = \{v \in \Ltwo :  \mathcal{R}v \in \mathcal{V} \}$
does not have a continuity requirement. The approximation space $\mW$ is closed under indefinite integration and multiplication and weak differentiation. This implies that all the functions that appear in the preconditioned CG method are in $\mW$.

Given a function that is piecewise analytic, we represent it by subdividing the interval $[-1,1]$ into $M+1$ subintervals, i.e., $[-1,x_1]\cup [x_1,x_2]\cup \cdots \cup [x_M,1]$, and representing the function by a Chebyshev expansion on each subinterval~\cite{pachon2009piecewise}. The Chebfun command that automatically determines the breakpoint locations and the polynomial degree to use on each subinterval is \texttt{p=chebfun(phi,'splitting','on')}. Any function that is computed during the CG method is automatically resolved in a piecewise fashion by Chebfun.

To solve for a piecewise smooth solution using spectral collocation or the ultraspherical spectral method, one has to construct a matrix that imposes the BVP operator on each subinterval along with continuity conditions at $x_i$ for $1\leq i\leq M$~\cite{driscoll2015rectangular}. In our preconditioned CG method the iterates $v_k$ belong to $\mW_0$, which is a space that contains functions that are not continuous. However, continuity on the approximate solutions $u_k = \mR v_k$ is implicitly imposed because $\mathcal{R}$ acts as a smoother.

The algorithms to compute the tasks of applying the preconditioner, the differential operator, and the projection operator are almost immediate from the algorithms in \cref{sec:analytic}. For example, if $\phi\in\mV$ and $x\in [x_m,x_{m+1}]$ for some $0\leq m\leq M$, then
\[
\mR \phi = \int_{-1}^x \phi(s) ds = \sum_{i=0}^{m-1} \int_{x_i}^{x_{i+1}} \phi(s) ds + \int_{x_m}^x \phi(s) ds.
\]
Therefore, to calculate the piecewise analytic function of $\mR \phi$ on $[x_m,x_{m+1}]$ one performs indefinite integration on $[x_m,x_{m+1}]$ using a recurrence relation~\cite[Sec.~8.1]{mason2002chebyshev} and adds to that the constant $\int_{-1}^{x_m} v(s)ds$ computed by applying Clenshaw--Curtis quadrature to each subinterval~\cite[Chap.~19]{trefethen2013approximation}.

\Cref{fig:piecewiseSpace} demonstrates the preconditioned CG method on three BVPs with piecewise smooth variable coefficients. The solutions of which have the same breakpoints as the variable coefficients. Since Chebfun automatically determines breakpoint locations for piecewise smooth functions~\cite{pachon2009piecewise}, our BVP solver automatically inherits this adaptivity. For piecewise continuous solutions we execute the same \texttt{pcg} command as in \cref{sec:analytic} without modification. As can be seen from the convergence theory in \cref{sec:CGtheory} and~\cref{fig:piecewiseSpace} (right), the convergence rate of the CG method is independent of the smoothness of the solution.

\begin{figure}
\centering
\begin{minipage}{.49\textwidth}
\begin{overpic}[width=\textwidth]{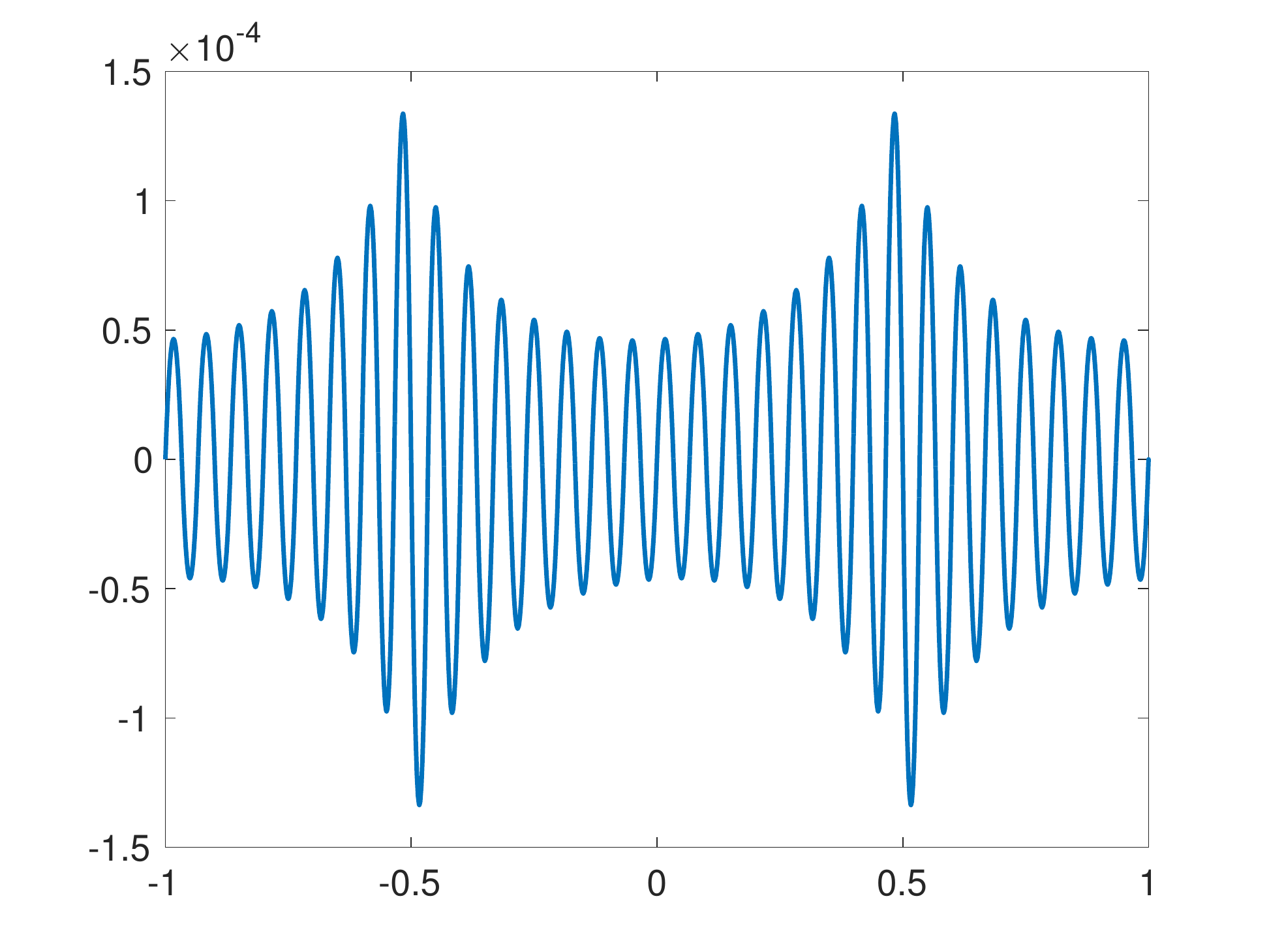}
\put(50,0) {$x$}
\end{overpic}
\end{minipage}
\begin{minipage}{.49\textwidth}
\begin{overpic}[width=\textwidth]{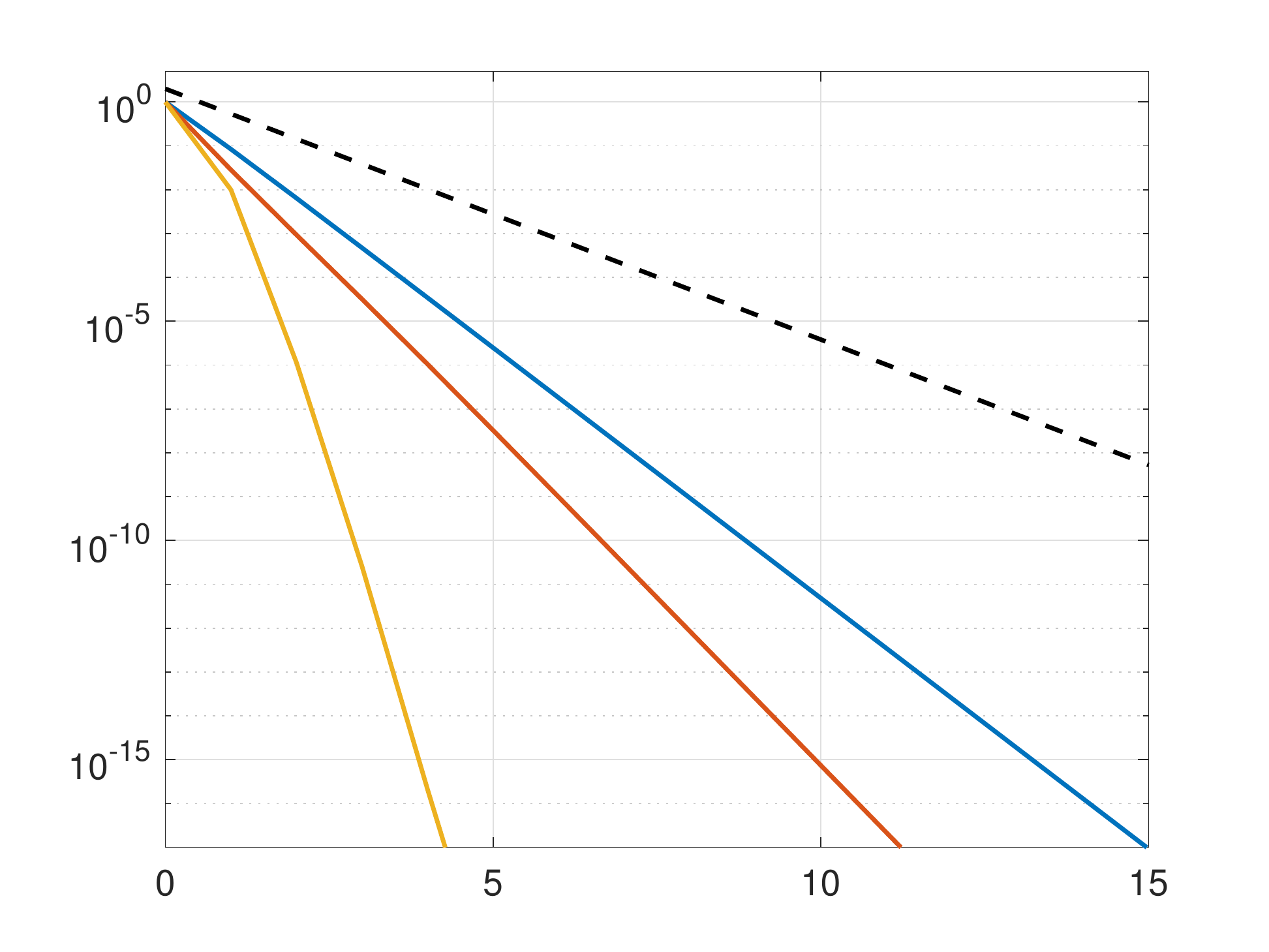}
\put(58,35) {\rotatebox{-40}{(E1)}}
\put(48,34) {\rotatebox{-47}{(E2)}}
\put(27,35) {\rotatebox{-70}{(E3)}}
\put(40,0) {$k$, iteration}
\put(0,20) {\rotatebox{90}{$\| u - u_k\|_{\mathcal{L}}/\|u\|_{\mathcal{L}}$}}
\put(40,62) {\rotatebox{-20}{$2\!\left(\frac{\sqrt{3}-1}{\sqrt{3}+1}\right)\!{}^{k}$}}
\end{overpic}
\end{minipage}
\caption{The preconditioned CG method for continuous functions that are piecewise analytic. Left: Solution to $\smash{-((1 + 2|\! \cos( \pi x) |) u')' = {\rm sign}(\cos(30\pi x))}$ with $u(\pm 1) = 0$. Right: The convergence of the preconditioned CG method for three BVPS with zero Dirichlet boundary conditions. (E1):  $\smash{-((1 + 2|\! \cos( \pi x ) | ) u')' = f}$ (blue line), (E2): $\smash{-( (1 + |\! \sin ( \pi x^2 ) | )u' )' + ( \tfrac{\pi}{4} )^2 | \cos (2 \pi x)| u= f }$ (red line), and (E3): $\smash{-u'' + 2(\tfrac{\pi}{4})^2 |\! \cos (20 \pi x )| u = f}$ (yellow line), where $\smash{f = (1+x^2)^{-1}}$. \Cref{cor:CGconvergence} gives the same convergence bound for these three BVPs (black dashed line).}
\label{fig:piecewiseSpace}
\end{figure}
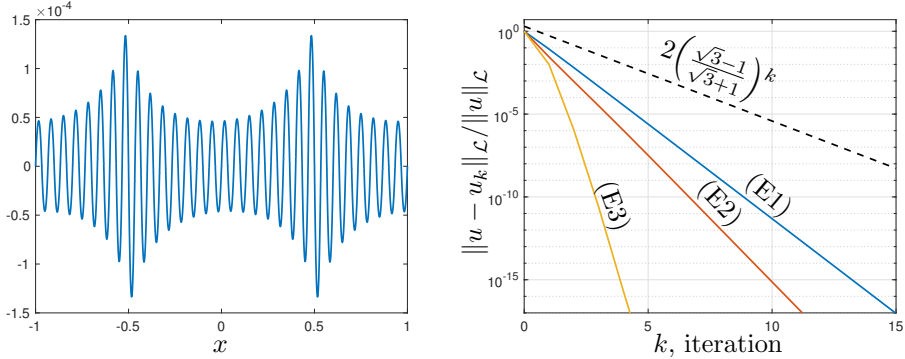

\section{Other Krylov-based methods}\label{sec:OtherKrylovBasedMethods}
The preconditioned CG method in \cref{sec:CGtheory} has provided us with an operator analogue of a Krylov subspace for solving~\cref{eq:PDE} (see~\cref{eq:preconditionedKrylov}). Two additional Krylov subspace methods for solving $Ax=b$ are MINRES (for symmetric linear systems)~\cite{paige1975solution} and GMRES (for general linear systems)~\cite{saad1986gmres}. In the matrix setting, MINRES and GMRES generate iterates by computing the best solution to $Ax=b$ from a Krylov subspace, as measured by the Euclidean norm of the residual, i.e.,
\begin{equation} \label{eq:minresdef}
  x_k = \argmin_{ y \in \mathcal{K}_k(A,b)} \| b - Ay\|_2, \qquad \mathcal{K}_k(A,b) = {\rm Span}\left\{b,Ab,\ldots,A^{k-1}b\right\},
\end{equation}
where $\|\cdot \|_2$ denotes the Euclidean norm of a vector.

Motivated by~\cref{eq:minresdef}, we set out to derive a MINRES and GMRES method for solving~\cref{eq:PDE} that constructs iterates so that
\begin{equation}
v_k = \argmin_{ p \in \mathcal{K}_k(\mathcal{T}, \mR^*f)} \|  \mR^* f -  \mathcal{T} p  \|,
\label{eq:GMRESproblem}
\end{equation}
where $\mathcal{R}$ is given in~\cref{eq:IntegralPreconditioner}, and $\mathcal{T}$ and $\mathcal{K}_k( \mathcal{T}, \mR^*f)$ are given in~\cref{eq:preconditionedKrylov}. The hope is that the iterates $u_k = \mR v_k$ converge to the solution $u$ of~\cref{eq:PDE}. In~\cref{eq:GMRESproblem}, we assume that $\mathcal{R}^*f\in \mW_0 = \{v\in\Ltwo : \mR v\in\Honezero\}$; otherwise, the ancillary problem in \cref{sec:GeneralRHS} is used to modify the right-hand side.

\subsection{The GMRES method for differential operators}\label{sec:GMRES}
The $k$th step of the GMRES method for solving $Ax=b$ computes an orthogonal basis for $\mathcal{K}_k(A,b)$ and then solves the least squares problem in~\cref{eq:minresdef} for $x_k$.
Analogously, our operator GMRES method computes an orthogonal basis for the Krylov subspace $\mathcal{K}_k(\mathcal{T}, \mR^*f)$. The orthogonal basis is computed via the decomposition\begin{equation}\label{eq:Arnoldi}
\mathcal{T} \mathcal{Q}_k = \mathcal{Q}_{k+1} \tilde{H}_k,
\end{equation}
where $\tilde{H}_k$ is a $(k+1) \times k $ upper Hessenberg matrix and $\mathcal{Q}_k$ is a quasimatrix with $k$ orthonormal columns.\footnote{A quasimatrix is a matrix whose columns are functions~\cite{stewart1998afternotes}. The quasimatrix has orthonormal columns if the columns are orthonormal with respect to the $L^2$ inner-product.} The decomposition is computed by an Arnoldi iteration on functions in $\Ltwo$ using modified Gram--Schmidt (see~\cref{alg:arnoldi}).

\begin{algorithm}[H]
\caption{Arnoldi iteration. Here, $\smash{\mathcal{T}}$ is the operator in~\cref{eq:preconditionedKrylov} and $\smash{\mathcal{R}^*f \in\mW_{0}}$. }\label{alg:arnoldi}
\begin{algorithmic}[1]
  \STATE{ $q_1 = \mathcal{R}^*f \big/ \|\mathcal{R}^*f\| $}
  \FOR{$k =2,\ldots,m$ }
  \STATE{$q_k = \mathcal{T}q_{k-1} $ }
  \FOR{$j = 1, \ldots, k -1$}
  \STATE{ $h_{j,k-1} = \<q_j, q_k\>$, $q_k = q_k - h_{j,k-1} q_j $ }
  \ENDFOR
  \STATE{ $h_{k,k-1} =\| q_k\| $, $q_k = q_k / h_{k,k-1}$}
  \ENDFOR
\end{algorithmic}
\end{algorithm}
Once an orthogonal basis for $\mathcal{K}_k(\mathcal{T}, \mR^*f)$ is computed by~\cref{alg:arnoldi}, the iterates from~\cref{eq:GMRESproblem} can be computed as follows:
\begin{equation*}
\argmin_{ p \in \mathcal{K}_k(\mathcal{T}, \mR^*f)} \|  \mR^* f -  \mathcal{T} p  \| = \argmin_{y \in \mathbb{R}^k}\| \mR^*f - \mathcal{T}\mathcal{Q}_k y \| = \argmin_{y \in \mathbb{R}^k} \left\| \|\mR^* f \|e_1 - \tilde{H}_k y \right\|_2 \ ,
\end{equation*}
which is a standard least squares problem that is typically solved by updating a QR factorization of $\tilde{H}_k$ at each iteration using Givens rotations \cite{van2003iterative}. We derive the following operator GMRES method for~\cref{eq:PDE}.

\vspace{-.2cm}

\begin{flushleft}
\begin{minipage}[t]{.99\textwidth}
\null
\begin{algorithm}[H]
\caption{The preconditioned GMRES method for~\cref{eq:PDE}, where $\smash{\mathcal{T}}$ is the operator in~\cref{eq:preconditionedKrylov},  $\smash{\mR^*f\in\mW_{0}}$ and $0<\epsilon<1$ is a tolerance on the norm of the residual.}\label{alg:gmres}
\begin{algorithmic}[1]
\FOR{$k =1,2,\ldots,$ }
\STATE{ Compute $\mathcal{Q}_{k+1}$ and $\tilde{H}_k$ in~\cref{eq:Arnoldi} using one step of~\cref{alg:arnoldi}  }
\STATE{Compute the QR factorization of $\tilde{H}_k$}
\STATE{ Solve $ \rho = \min_{y}\left\| \|\mR^* f \|e_1 - \tilde{H}_k y  \right\|$  }
\IF{ $\rho < \epsilon$  }
\STATE{ $v = \mathcal{Q}_{k} y$}
\STATE{ $u = \mR v$}
\STATE{ \bf{stop iteration} }
\ENDIF
\ENDFOR
\end{algorithmic}
\end{algorithm}
\end{minipage}
\end{flushleft}

\vspace{.5cm}

Unlike CG, the computational and storage costs of GMRES grows with the number of iterations. To avoid excessive storage costs, the GMRES method is usually restarted after $m$ iterations for some integer $m$, i.e., $v_m$ becomes an initial guess for a new GMRES method. The convergence behavior of the GMRES method is difficult to fully characterize and the statements that can be presented for convergence are analogous to those for the matrix GMRES method~\cite[Chap.~6]{van2003iterative}. \Cref{fig:gmres} (left) shows the convergence of the preconditioned GMRES on the BVP \linebreak $\smash{-(e^{x}u')' + u' - 10 u = \sin (30\pi x), u(\pm 1 ) = 0}$ for different restarts. As observed in the matrix case the convergence can deteriorate with too frequent restarts, though iterates after restarting are computed more efficiently.

The operator GMRES method is implemented in Chebfun in the \texttt{gmres} command and has precisely the same realizations as the operator CG method (see \cref{sec:CG1dAlgorithm}).
\begin{figure}
\vspace{0.3cm}
\centering
\begin{minipage}{.49\textwidth}
\begin{overpic}[width=.95\textwidth]{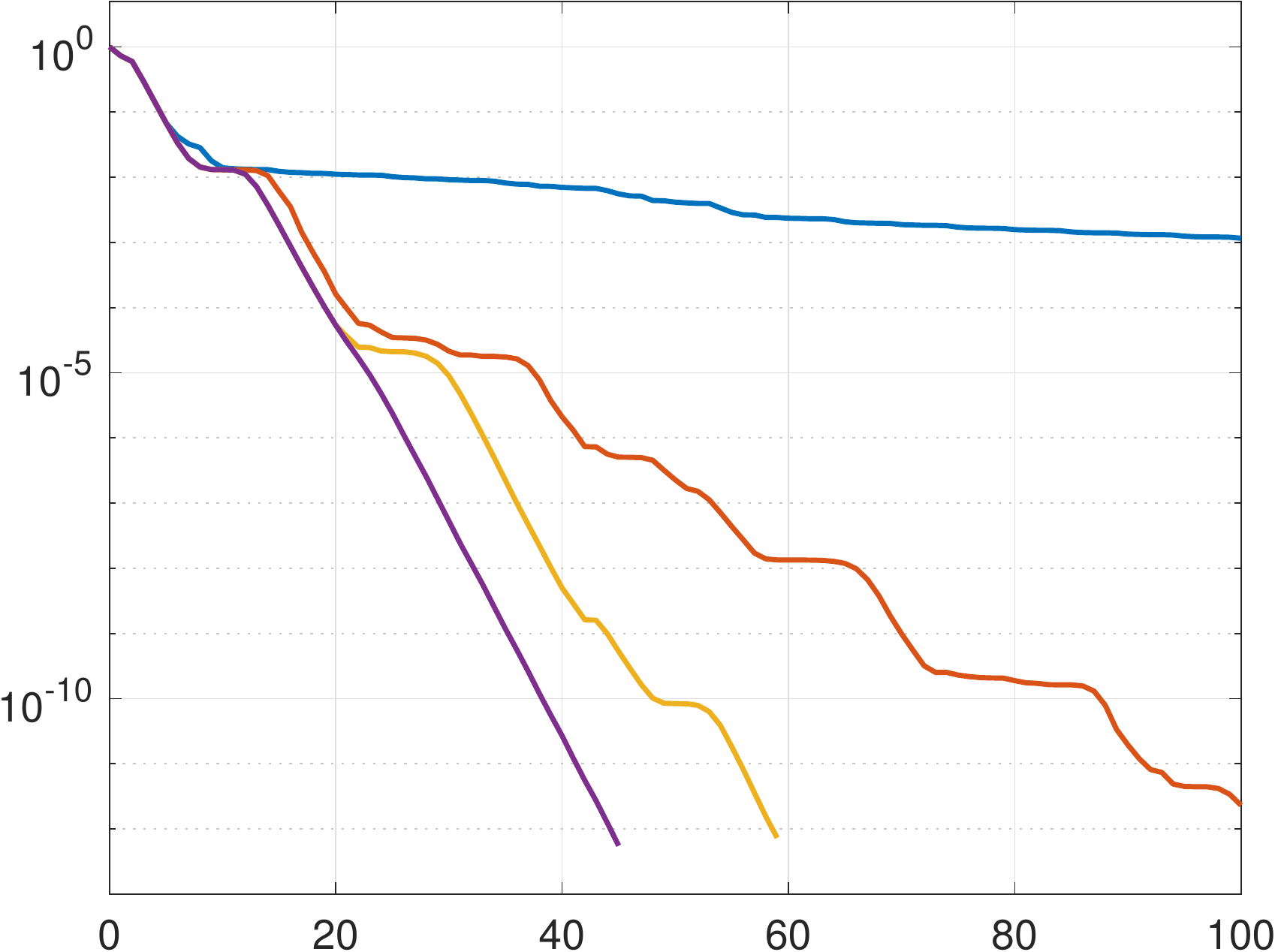}
\put(40,-4) {$k$, iteration}
\put(-7,15) {\rotatebox{90}{$\| \mR^*f - \mathcal{T} v_k \|/\| v_k \|$}}
\put(-4,77) {$-(e^{ x}u')' + u' - 10 u  = \sin (30\pi x), u(\pm 1) = 0$}
\put(38.5,29.5) {\rotatebox{-63}{$m= 100$}}
\put(42,35) {\rotatebox{-53}{$m = 20$}}
\put(53,40) {\rotatebox{-30}{$m = 10$}}
\put(70,58) {\rotatebox{0}{$m = 5$}}
\end{overpic}
\end{minipage}
\begin{minipage}{.49\textwidth}
\centering
\begin{overpic}[width=.95\textwidth]{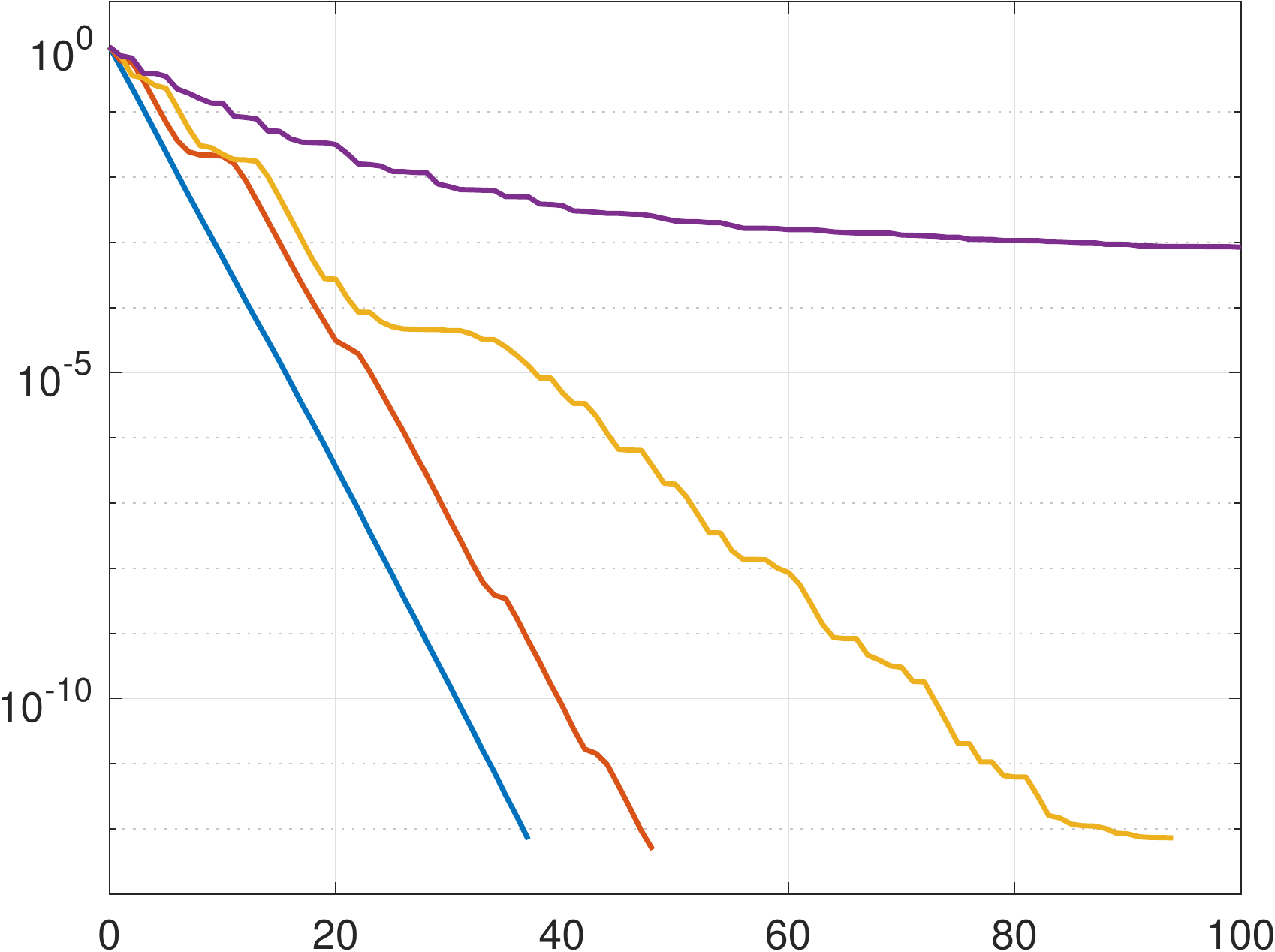}
\put(40,-4) {$k$, iteration}
\put(-7,15) {\rotatebox{90}{$\| \mR^*f - \mathcal{T} v_k \|/\| v_k \|$}}
\put(4,77) {$-(e^{x}u')' - \lambda u = \sin (30\pi x), u(\pm 1) = 0$}
\put(31.5,29.5) {\rotatebox{-63}{$\lambda = 1$}}
\put(41,30) {\rotatebox{-63}{$\lambda = 10$}}
\put(53,40) {\rotatebox{-45}{$\lambda = 100$}}
\put(70,58) {\rotatebox{0}{$\lambda = 1000$}}
\end{overpic}
\end{minipage}
\caption{Left: Convergence of the restarted GMRES method with restarts every $m$ iterations for $m =  5$ (blue line), $m=10$ (red line), $m=20$ (yellow line), and  $m=100$ (purple line). Right: Convergence of the MINRES method for $-(e^{x}u')' - \lambda u = \sin (30\pi x)$ with $u(\pm 1) =0$ with $\lambda = 1$ (blue line), $\lambda = 10$ (red line), $\lambda=100$ (yellow line), and $\lambda=1000$ (purple line). The quality of the indefinite integral preconditioner in~\cref{eq:IntegralPreconditioner} is reduced as $\lambda$ increases. }
\label{fig:gmres}
\end{figure}
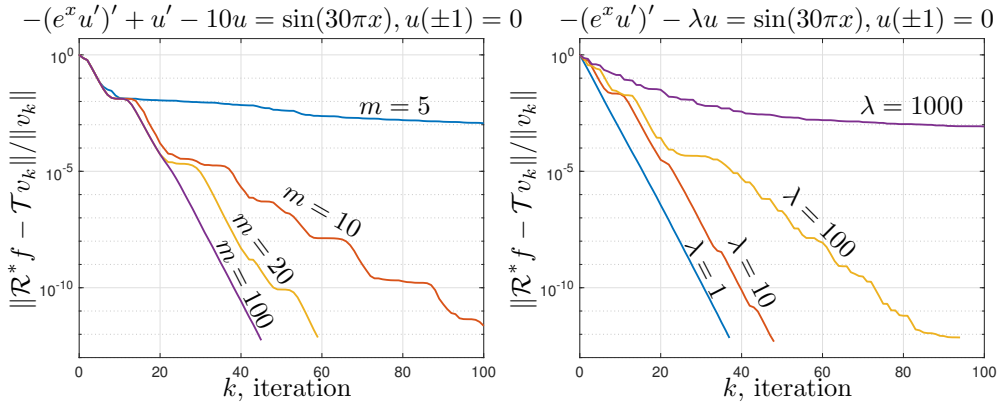

\subsection{The MINRES method for differential operators}\label{sec:MINRES}
MINRES can be described as a special case of GMRES that applies when the linear system is symmetric. In that situation, the matrix $\tilde{H}_k$ reduces to a tridiagonal matrix and a Lanczos procedure is used instead of an Arnoldi iteration~\cite{paige1975solution}.  For self-adjoint second-order differential operators, it is analogous.  Thus, our operator MINRES method is a GMRES method without restarts that exploits the fact that the operator is self-adjoint. An optimized implementation of MINRES notes that $\mathcal{Q}_k$ and $\mathcal{H}_k$ in~\cref{eq:Arnoldi} do not need to be stored and that the solution $y$ can be efficiently updated from previous iterates. The convergence properties of operator MINRES are analogous to the convergence behavior of MINRES for solving linear systems.

We have implemented MINRES in Chebfun in the \texttt{minres} command, which has the same practical realizations as the CG method (see \cref{sec:CG1dAlgorithm}). \Cref{fig:gmres} (right) shows the convergence of the preconditioned MINRES method on the family of BVPs $-(e^{x}u')' - \lambda u = \sin (30\pi x)$ with $u(\pm 1 ) = 0$ for different values of $\lambda$.

\section*{Conclusion}
Operator analogues of the CG method, MINRES, and GMRES are derived for solving BVPs on $(-1,1)$ that employ operator-function products. An operator preconditioner ensures that only a finite number of Krylov iterations are necessary to compute an approximate solution, and an orthogonal projection operator guarantees that the computed Krylov subspace imposes the boundary conditions of the BVP\@. The resulting iterative solvers are able to compute solutions from $\Honezero$ and are competitive BVP solvers when a fast operator-function product is available.

\section*{Acknowledgements}
The authors tried several different approaches to developing an operator analogue of Krylov methods, and we thank Jiahao Chen and Alan Edelman for interesting discussions on several unfruitful attempts. We are grateful to Nick Trefethen for popularizing the idea of operator analogues of Krylov subspace methods in 2012, and we thank Jared Aurentz, Sheehan Olver, and Marcus Webb for sharing their thoughts on the topic. We also thank David Bindel, Toby Driscoll, Dan Fortunato, Hadrien Montanelli, and Heather Wilber for comments and reactions on different versions of this work.

\bibliographystyle{siam}
\bibliography{ContinuousCG}

\end{document}